\documentclass[12pt]{amsart}     

\usepackage{amsfonts}
\usepackage{amsmath}
\usepackage{graphicx}
\usepackage{amsfonts} \usepackage{amsmath, amsthm}
\usepackage{epsfig} \usepackage{verbatim}
\usepackage{amssymb, amsthm, amsmath, bm, tikz, enumerate, color}
\usepackage{ bm, tikz, enumerate}
\usepackage{amsfonts}
\usepackage{amsmath}
\usepackage{graphicx}
\usepackage{amsfonts} \usepackage{amsmath, amsthm}
\usepackage{epsfig} \usepackage{verbatim}
\usepackage{amssymb, amsthm, amsmath, bm, tikz, enumerate, color}
\usepackage{cite}
\usepackage{todonotes}
\usepackage{xcolor}
\usepackage{constants}
\usetikzlibrary{spy}

\newtheorem{theorem}{Theorem}

\newtheorem{lemma}{Lemma}

\newtheorem{remark}[theorem]{Remark}

\newcommand{\dist}{\mathop{\mathrm{dist}}\nolimits}

\newconstantfamily{Ns}{symbol=N}

\newcommand{\boundellipse}[3]% center, xdim, ydim
{(#1) ellipse (#2 and #3)
}

\newcommand\restr[2]{{% we make the whole thing an ordinary symbol
  \left.\kern-\nulldelimiterspace % automatically resize the bar with \right
  #1 % the function
  \vphantom{\big|} % pretend it's a little taller at normal size
  \right|_{#2} % this is the delimiter
  }}

\author{Margaret Brown}
\address{Department of Mathematics, Penn State University,
State College, PA 16802, USA and 
Department of Mathematics, University of Maryland, College Park, MD 20741, USA}
\email{mlb6635@psu.edu}
\author{P\'eter N\'andori}
\address{Department of Mathematical Sciences, Yeshiva University, New York, NY, 10016, USA}
\email{peter.nandori@yu.edu }
\title
{Statistical properties of type D dispersing billiards}
\date{}

\begin{document}

\maketitle

\begin{abstract}
We consider dispersing billiard tables whose boundary is piecewise smooth and the free flight
function is unbounded. We also assume there are no cusps. Such billiard tables are called type D 
in the monograph of Chernov and Markarian \cite{CM06}.
For a class of non-degenerate type D dispersing billiards, 
we prove exponential decay of correlation and several other statistical properties.
\end{abstract}

\section{Introduction}

%Dispersing billiards of type D are informally defined as follows. 

Consider a collection of disjoint open sets on the  torus
$\mathbb T^2 = \mathbb R^2 /\mathbb Z^2$
(called scatterers in the sequel) with piecewise $C^3$ boundary 
which are locally convex with bounded from below curvature at regular points. We assume that there are no cusps.
To define the Sinai billiard flow \cite{Si70},
let a point particle fly freely with constant speed on the complement of the scatterers (called the billiard table) and be
subject to elastic collision upon reaching their boundaries. 
Depending on the geometry of the scatterers, the free flight time may or may not be bounded. 
A partial classification of dispersing billiard tables is given by \cite{CM06} as follows:
assume first that the boundary of the billiard table is $\mathcal C^3$. If the free flight is bounded, then the table is of type A, otherwise of type B. 
Now assume that the boundary of the billiard table is only piecewise $\mathcal C^3$.
Points where the boundary is $\mathcal C^3$ are called regular. The finitely many 
non-regular points are called corner points.
If the free flight is bounded, then the table is of type C, otherwise of type D. (In case of cusps, type E and F.)
Statistical properties were first proven for type A billiards 
(see the central limit theorem in \cite{BS81, 
BSCh90, BSCh91}, exponential decay of correlations \cite{Y98}). 
Next, type B tables were also extensively studied (see \cite{B92, SzV07, C06}).
Although there are early works for types C and D \cite{BSCh91}, the more recent theory 
(such as the construction of Young towers \cite{Y98}) was not studied in these classes until recently. 
For type C billiard tables, \cite{DST13} proves the $m$-step expansion estimate, which together with other estimates 
(that can be proved
as in type A) yield the statistical properties mentioned above.
There are fewer results available in types D-F (in fact, these classes are 
labelled as "hard" in \cite{CM06})

It is standard that long free flights are only possible after a collision in
a small vicinity of finitely many points, which we call boundary points of corridors. 
We now distinguish two classes of type D billiard tables:
if all boundary points of corridors are regular, we say that the billiard table is of type 
D1, otherwise of type D2.
The main result of the present work can be informally stated as follows.
\begin{theorem}
\label{thm0}
Consider a billiard table of type D1 or
type D2, in which case we also require that assumptions
(A1) and (A2) hold. Then the correlation of bounded dynamically H\"older observables decay exponentially fast and the central limit theorem
holds for such observables.
\end{theorem}
The precise definitions 
are given in Section \ref{sec:pre}.
We note that some conditional results are available in the literaure, 
see e.g. the condition on complexity in \cite{BSCh90}.
Up to our best knowledge that condition on complexity 
is not known to hold generically and in fact is not verified
for any specific billiard table. 
Our conditions (A1), (A2)
hold on an open and dense set of billiard tables and furthermore given any billiard table
it is easy to check whether they hold as they only depend on the boundary
points of the corridors.

The rest of this paper is organized as follows. In Section \ref{sec:pre} we collect the necessary background information needed in this work.
None of the results of Section \ref{sec:pre} are new.
In Section \ref{sec:res} we state our main technical theorems. Theorem \ref{thm1} implies Theorem \ref{thm0} in type D1. 
Theorem \ref{thm2} implies Theorem \ref{thm0} in type D2.
%Finally, Theorem \ref{thm3} implies that our conditions are generic. 
Section
\ref{sec:thm1} contains the proof of Theorem \ref{thm1}. 
%This section is short, as
%Theorem \ref{thm1} is a relatively straightforward extension of the main result of \cite{DST13}, combined with an estimate in type B \cite{CM06}.
Section
\ref{sec:thm2} contains the proof of Theorem \ref{thm2}. This proof is substantially more complex then that of  Theorem \ref{thm1}
as we need a careful study of the geometry of long free flights in case 
of type D2 configurations. The proof of an important lemma is postponed to 
Section \ref{sec:lemma9}. In Section 
\ref{sec:thm3}, we prove that the conditions (A1), (A2) hold on an open and
dense set of billiard tables.

We mention that very recently there has been an increasing interest in the detailed 
description of possible orbits in infinite corridors in cases of hyperbolic billiards
\cite{ABB19, PT20} and in some similar hyperbolic systems with singularities \cite{DDS20}.

%If it is bounded,
%the billiard is called of 
%type C otherwise it is called type D \cite{CM06}.

\section{Preliminaries}
\label{sec:pre}

Here we review the preliminaries needed for our work. All results in this section are known, see
\cite{C99,ChD09,CM06}. More specific references will be given for the most important statements.

\subsection{Billiards of type D}

Let $\mathbb T^2$ be the $2$-torus and $\mathcal D \subset \mathbb T^2$ be a
dispersing billiard table. 
That is,
the complement of $\mathcal D$ consists of finitely many (say $d$) connected components $\mathcal B_i$ (called scatterers). 
For convenience we also label the scatterers. Each scatterer 
$i = 1,...,d$ is bounded by a finite union of curves $\Gamma_{i,j}$, $j=1,...,J_i$. It is assumed that $\Gamma_{i,j}$
is a $\mathcal C^3$ curve, that is there is a $\mathcal C^3$ function $f_{i,j}: [0,1] \rightarrow \mathbb T^2$ which is a 
bijection between $[0,1]$ and $\Gamma_{i,j}$. 
Furthermore, $f_{i,j}(1) = f_{i,j+1}(0)$ where $j+1$ is interpreted 
modulo $J_i$ (that is, $f_{i,J_i}(1) = f_{i,1}(0)$). 
The endpoints of $\Gamma_{i,j}$ are called corner points, all other
points of $\Gamma_{i,j}$ are regular points. 
We require that one of the first three one-sided derivatives at $f_{i,J_i}(1)$ differ from the corresponding derivative at
$ f_{i,1}(0)$, that is no regular point is labelled as corner point.
We also require that the curvature of $\Gamma_{i,j}$ is positive 
with uniform upper and lower bounds 
%($\kappa_-$ and $\kappa_+$)
at all regular points. 
The orientation of $\Gamma_{i,j}$ is assumed to be so that when following $\Gamma_{i,1}$, 
$\Gamma_{i,2}$, ..., $\Gamma_{i,J_i}$, we follow clockwise orientation and $\mathcal D$ is to the left hand
side. The region enclosed by $\Gamma_{i,1}, ..., \Gamma_{i, J_i}$ (a subset of 
$\mathbb T^2 \setminus \mathcal D$) is one scatterer. 
If the boundary of the scatterer $i$ is 
$\mathcal C^3$ smooth, i.e. does not contain corner points, then the scatterer is necessarily strictly 
convex
and $J_i= 1$.
We also assume no cusps, that is the tangent lines of
$\Gamma_{i,j}$ and $\Gamma_{i,j+1}$ have an angle of at least $\alpha_0$ at their common endpoint $f_{i,j}(1)$.
%(this assumption also guarantees that the corner points are non regular). 
The interiors of $\Gamma_{i,j}$
are disjoint for all $i,j$. Furthermore, at each corner point exactly two curves meet 
and their angle is bounded from below 
by some positive constant.
%, as well as the minimal distance between scatterers is also bounded away from zero.
Any billiard table satisfying these assumptions is called admissible.

Given a corner point, let $\gamma$ be the angle between the two half tangent lines at it, 
measured at the interior of $\mathcal D$. 
The admissible property implies that $\gamma \neq 0$ for all corner points 
(the case $\gamma = 0$ is called a cusp). 
We say that the corner point is {\it convex} if $ 0 < \gamma \leq \pi$. Note that $\gamma = \pi$
is possible, in this case we assume that either the second or the third derivatives on the 
two sides of the corner points differ. We say that the corner point is {\it concave} if 
$ \pi < \gamma < 2\pi$ (noting that $\gamma = 2 \pi$ is impossible due to local convexity of the scatterers
at regular points). See Figure \ref{fig0}.

\begin{figure}
\begin{center}

\begin{tikzpicture}[scale=0.5]

\fill[fill=lightgray] plot[domain=0:2*pi] (xy polar cs:angle=\x r,radius= {2-2*sin(\x r)});
\draw[thick,domain=0:2*pi,samples=200,smooth] plot (xy polar cs:angle=\x r,radius= {2-2*sin(\x r)});
%\node at (2,1) {\scriptsize $r=2-2\sin\theta$};
%\draw[->] (-4,0) -- (4,0);
%\draw[->] (0,-5) -- (0,2);

\end{tikzpicture}
\hspace{3 cm}
\begin{tikzpicture}[scale=4.5, rotate = 10.5]

%\draw[fill = gray!50] (-1,1) arc (150:180:1) arc (30:60:1);
\draw[thick, fill = gray!50] (1,0) arc (30:60:1) arc (100:130:1);

\end{tikzpicture}
\caption{A scatterer with a convex corner point (left) and a concave corner point (right)} \label{fig0}
\end{center}
\end{figure}
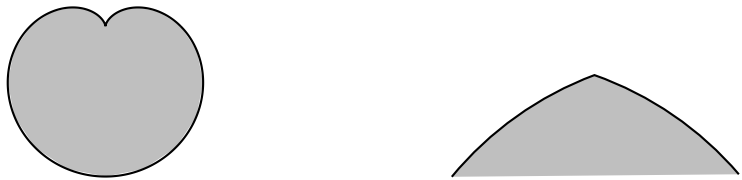

%the angle between the two half tangent lines at it, 
%measure at the interior of $\mathcal D$
%of the table is bigger than $\pi$. Otherwise (i.e. 
%when the angle is bigger than $\pi$), the corner point is
%called concave (in this case the scatterer is not locally convex). 

Given two admissible 
billiard tables $\mathcal D_1, \mathcal D_2$
with the same combinatorial data (that is the same number of scatterers $d$ and the same number
of smooth pieces $J_i$, $i=1,..,d$), we define their distance as 
$$
d(\mathcal D_1, \mathcal D_2)  = \inf_{\{ f^1\} ,\{ f^2\} } \max_{i,j} d_{\mathcal C^3} (f^1_{i,j}, f^2_{i,j}),
$$
where the infimum is taken over admissible parametrizations, i.e. collections of $\mathcal C^3$ 
functions $f^k_{i,j}$ so that $f^k_{i,j}$ is a bijection between $[0,1]$ and $\Gamma^k_{i,j}$ where $k=1,2$ indicates
the two billiard tables.
This makes the set of labelled admissible billiard tables
with given combinatorial data a metric space $\bm D_{d, J_1,...,J_d}$.
%Note that $\bm D_{d, J_1,...,J_d}$ is not complete. Indeed, for example
%we can consider a sequence of tables
%where two scatterers approach each other (say without corner points). Then the limiting billiard table would have
%a cusp. However the subspace of $\bm D_{d, J_1,...,J_d}$ only consisting of tables with fixed
%$\alpha_0, \kappa_-, \kappa_+$ (uniformly for all tables), is complete. 
Let $\bm D$ denote the space of all (labelled) admissible billiard tables, that is 
$\bm D =  \cup_{d, J_1,...,J_d} \bm D_{d, J_1,...,J_d}$. The space $\bm D$ is also a metric space
by defining the distance between two tables of different 
combinatorial data to be infinite (mind the labelling).

The billiard dynamics on a fixed admissible billiard table $\mathcal D$ prescribes the motion of 
a point particle that flies with constant speed $1$ in a given direction
 $v$ until it reaches the boundary $\partial \mathcal D$,
where it undergoes an elastic collision (meaning the angle of reflection equals the angle of incidence). 
The phase space of the billiard flow is
$\Omega = \mathcal D \times \mathcal S^1 / \sim $, where $\sim$ means identifying pre-collisional and post-collisional
data (that is, if $q \in \partial \mathcal D$ is a regular point, then $v$ and $-v$ are identified unless $v$ is tangent to $\partial \mathcal D$ at $q$. We will discuss
the case of corner points in more detail later). 
We use the notation $(q,v) \in \Omega$ with $v$ being the velocity vector.
We %say that $q\in \mathcal D$ is the configurational component of the point $(q,v)$ and denote $\Pi_{\mathcal D} (q,v) = q$ and also 
say that $\mathcal D$ is the configuration space and $q = \Pi_{\mathcal D} (q,v)$ 
is the configurational
component of $(q,v)$. 
The billiard flow is denoted by $\Phi^t : \Omega \mapsto \Omega$ for every $t \in \mathbb R$. 

Note that the dynamics may not be well defined upon reaching a corner point. Such trajectories have
Lebesgue measure zero, so the definition of $\Phi$ on this set is irrelevant for physical properties. 
It is convenient, though, to
define the flow to be possibly multi-valued upon reaching a corner point, corresponding to possible
limit points of nearby regular orbits. One way of defining the flow is as follows. First, we say that a collision
is improper if the trajectory can be approximated by trajectories missing the collision. In the case of smooth 
scatterers, an improper collision is the same as a grazing collision. In the case of a concave corner point, we
may have an improper collision which is not grazing (such as a horizontal flight touching the 
corner point on the right of Figure \ref{fig0}). A proper collision is a collision that is not improper.
For example, a vertical flight hitting the corner on the right of Figure \ref{fig0} is proper,
and nearby regular trajectories have two possible continuations. All trajectories hitting a convex corner point
are proper. Furthermore, we may have a sequence of short flights near the convex
corner point
(also known as corner sequence), but the number of short collisions (the length of the corner
sequence) is bounded due to the assumption that there are no cusps (see \cite[Section 9]{C99}).
Now given a point $(q,v)$, put 
$\tau(q,v) = \inf \{ t>0: \Pi_{\mathcal D} \Phi^t(q,v) \in \partial \mathcal D\}$. Now assuming that 
$\lim_{t \nearrow \tau(q,v)} \Pi_{\mathcal D} \Phi^t(q,v)$ is a corner point $\tilde q \in \partial \mathcal D$, 
we define
$\Phi^{\tau(q,v)}(q,v)$ as 
$$\lim_{\varepsilon \searrow 0} \lim_{q' \rightarrow q, v' \rightarrow v} \Phi^{\tau(q,v) + \varepsilon}(q',v')$$
where 
$q',v'$ are points that can only experience collisions at regular points up to time $\tau(q,v) + \varepsilon$
and
the second limit is to be interpreted as the set of all possible limit points.
With this definition, $\Phi^{\tau(q,v)}(q,v)$ can take one or two values. 
This is trivial in case of concave corner points; for convex points see \cite[Section 2.8]{CM06}.
The flow $\Phi^t$ preserves the Lebesgue measure
$\nu$ on $\Omega$ (we assume by normalization that $\nu$ is a 
probability measure). 

We will also study the billiard map. Let
$\mathcal M$ be a cross-section of post-collisional points. Then $\mathcal M$ can be identified with 
a union of cylinders and rectangles. 
For any curve $\Gamma_{i,j}$, we define $\mathcal M_{i,j} = [a_{i,j},b_{i,j}] \times [-\pi/2, \pi/2]$ where 
$b_{i,j} - a_{i,j} = |\Gamma_{i,j}|$ and the intervals $[a_{i,j},b_{i,j}]$ are disjoint. If the scatterer $i$ is smooth
(in this case necessarily $J_i = 1$), then
we identify the endpoints of the interval $[a_{i,1},b_{i,1}] $, so $\mathcal M_{i,1}$
becomes a cylinder.
Finally, we put 
$\mathcal M = \cup_{i,j} \mathcal M_{i,j}$. Coordinates 
in $\mathcal M$ are denoted by $(r, \varphi)$:
$r$ is arclength parameter along the boundary of the scatterer in clockwise direction;
$\varphi$ is 
the angle of the postcollisional velocity and the normal to $\mathcal D$ at $q$ pointing
into $\mathcal D$. The angle 
$\varphi$ is 
also measured in the clockwise direction
with $\varphi \in [-\pi /2 , \pi /2 ]$ (see \cite[Figure 2.14]{CM06}).
The billiard map is denoted by $F: \mathcal M \rightarrow \mathcal M$. 
It preserves the physical invariant measure $\mu$ defined by $d \mu = C_{\mu} \cos \varphi dr d\varphi$, where 
$C_{\mu}$ is a normalizing constant
($\mu$ is obtained as the projection of $\nu$ to the Poincar\'e section). 
The flow is now a suspension over the base map $F$ with roof function
$\tau$, which is the free flight time.
Note
that $F$ can be multivalued at points when the next collision is at a corner point. The special case of unbounded free flight near a corner 
point will be discussed in the next section.

For ease of notation, we will identify $\mathcal M$ with a subset of $\Omega$ in the natural way. 
For example, we will write $\Pi_{\mathcal D} x$
for $x \in \mathcal M$. 

\subsection{Structure of corridors}

Next we study corridors. We say that an admissible billiard table has {\it infinite horizon} if 
the free flight is unbounded. In this case, there are finitely many "corridors". A corridor $H$ by definition
is a direction $v= v_H \in [0, \pi)$ 
and a 
connected subset $Q_H$ of $\mathcal D$ with non-empty interior
\begin{equation}
\label{eq:QH}
Q_H=\{ q \in \mathcal D: \forall t \in \mathbb R: q+tv \in \mathcal D \}.
\end{equation}
There are only finitely many corridors 
(see \cite[Exercise 4.51]{CM06}). Let us say that an admissible billiard table is of type D1 if all corridors are bounded by 
grazing orbits at regular points (such orbits are necessarily periodic). In other words, a billiard table is regular
if no corner point is in the corridors. 
If the billiard table is of type D but not of type D1, then 
we call it type D2.
We say that an admissible billiard table is simple if for all corridors $H$, 
$B_H := \partial Q_H \cap \partial \mathcal D$ consist of exactly $2$ points, one on both 
sides of $Q_H$, that is $B_H=\{ q_{H,l}, q_{H,r} \}$. Here, $l$ and $r$ stand for left and right points, when viewed
from the direction $v$. 
For simple billiard tables, we consider the four points in $\mathcal M$, whose trajectory up to the next
collision projects onto $\partial Q_H$ in the configuration space. The set of these four points is denoted
by
\begin{equation}
\label{eq:horizonbd}
A_H = \{ (r_{H,l,1},\varphi_{H,l,1}), (r_{H,l,2},\varphi_{H,l,2}), (r_{H,r,1},\varphi_{H,r,1}), (r_{H,r,2},\varphi_{H,r,2})\}.
\end{equation}
We will say that the elements of $A_H$ are boundary points of the corridor $H$.
Note that if $q_{H,s}$ is a regular point (for $s=l,r$), then $r_{H,s,1} = r_{H,s,2}$ corresponds to a vertical
line segment in the interior of $\mathcal M_{i,j}$. On the other hand, if $q_{H,l}$ is a corner point, then
$r_{H,l,1}$ corresponds to the right side of $\mathcal M_{i,j}$ and $r_{H,l,2}$ corresponds to the left side of
$\mathcal M_{i,j+1}$ 
(and vice versa for the right boundary:
if $q_{H,r}$ is a corner point, then
$r_{H,r,2}$ corresponds to the right side of $\mathcal M_{i',j'}$ and $r_{H,r,1}$ corresponds to the left side of
$\mathcal M_{i',j'+1}$).
In this case and with a slight abuse of notation, we will also say that the elements of $A_H$ are corner points.
Note that whenever $q_{H,s}$ is a corner point, it is necessarily concave.
See Figure \ref{fig1} for a typical arrangement in the case that both sides are bounded by
a corner point (for simplicity, we depict $i=j=j'=1$, $i'=2$, $v$ is horizontal, so by convention pointing to the right).
The figure represents a part of the scatterer configuration lifted from $\mathbb T^2$ to $\mathbb R^2$.
The point $q_{H,r} = \Gamma_{1,1} \cap \Gamma_{1,2}$ is the corner point on the bottom left as well as
the bottom right of the figure. 
The two corresponding signed angles $\varphi_{H,r,k}$, $k=1,2$ are between the dashed lines
(normals to the curves) and the lower dotted line. 
Similarly, observe the left boudary of the corridor on 
the top part of the figure. In this paricular case, 
we have $\varphi_{H,r,1} >0$, $\varphi_{H,r,2} <0$,
$\varphi_{H,l,1} <0$,
$\varphi_{H,l,2} >0$ although these signs may be different for other corridors bounded by two corner points.
%($\varphi_{H,l,2} $ is the small angle on the top middle scatterer, not indicated for better visibility).
Note that 
\begin{equation}
\label{eq:criticalorbit}
F(r_{H,r,1},\varphi_{H,r,1}) = (r_{H,r,2},-\varphi_{H,r,2})
\end{equation}
which corresponds to an improper collision. According
to our definition, $\Phi^{\tau(r_{H,r,1},\varphi_{H,r,1})}
(r_{H,r,1},\varphi_{H,r,1})$ takes two values:
$(r_{H,r,2},-\varphi_{H,r,2})$ and $(r_{H,r,1},\varphi_{H,r,1})$
(where we identified $\mathcal M$ with a subset of $\Omega$).
This corresponds to an improper collision and so by our
terminology \eqref{eq:criticalorbit} holds.
There are two possible types of free flights from 
regular points close to $(r_{H,r,1},\varphi_{H,r,1})$. One possibility is 
a flight of bounded length, terminating on $\Gamma_{1,1}$. For such points $(q,v)$,
$F(q,v)$ is close to $(r_{H,r,2},-\varphi_{H,r,2})$. The other possibility is a very long flight
in the corridor which eventually terminates on $\Gamma_{2,2}$. For such points $(q,v)$,
$F(q,v)$ is close to $(r_{H,l,2},-\varphi_{H,l,2})$.
The local geometry of such orbits will be 
studied more carefully later.

%for small $\varepsilon >0$: one point flies upward in the direction
%of $-\varphi_{H,r,2}$ past the collision at $q_{H,r}$, the other
%one continues vertical flight. This trajectory is the same

%the first one corresponds to the limit of nearby orbits all colliding in bounded time on $\Gamma_{1,1}$ and the 
%second one corresponds to nearby orbit with long free flights crossing the horizon and eventually colliding on 
%$\Gamma_{2,2}$.

%Noting that the map $F$ can be multivalued, we say that a point $x$ is non periodic if $x \notin F^n(x)$ for all $n \geq 1$. 

\begin{figure}
\begin{center}

\begin{tikzpicture}[scale=4.5]

%UPPER EDGE

\draw (1,0.5) arc (290:335:0.5);
\draw (0.9236,0.75) arc (195:200:3);

\draw (2.2,0.5) arc (290:335:0.5);
\draw (2.1236,0.75) arc (195:200:3);

\draw (-0.2,0.5) arc (290:335:0.5);
\draw (-0.2764,0.75) arc (195:200:3);

%%LOWER EDGE
\draw (0,0) arc (135:159:1);
\draw (0.073,-0.375) arc (0:22:1);

\draw (2,0) arc (135:159:1);
\draw (2.073,-0.375) arc (0:22:1);

%TREJACTORIES

\draw[densely dotted] (-0.4,0.5) -- (2.4,0.5) ;
\draw[densely dotted] (-0.2,0) -- (2.2,0);

\draw[dashed] (0,0) -- (0.2, 0.08);
\draw[dashed] (1,0.5) -- (0.6,0.4);%(0.8, 0.45);

\draw[dashed] (-0.2,0.5) -- (-0.15, 0.3);
\draw[dashed] (2,0) -- (1.85, 0.15);

\node[right] at (0.1,0.02) {$\varphi_{H,r,1}$};
\node[right] at (1.7,0.02) {$\varphi_{H,r,2}$};
\node[right] at (-0.2,0.45) {$\varphi_{H,l,1}$};
\node[right] at (0.5,0.45) {$\varphi_{H,l,2}$};

\node[right] at (0.05,-0.3) {$\Gamma_{1,2}$};
\node[right] at (-0.4,-0.3) {$\Gamma_{1,1}$};
\node[right] at (2.05,-0.3) {$\Gamma_{1,2}$};
\node[right] at (1.6,-0.3) {$\Gamma_{1,1}$};

\node[right] at (0.06,0.7) {$\Gamma_{2,1}$};
\node[right] at (-0.45,0.7) {$\Gamma_{2,2}$};

\node[right] at (1.26,0.7) {$\Gamma_{2,1}$};
\node[right] at (0.75,0.7) {$\Gamma_{2,2}$};

\node[right] at (2.46,0.7) {$\Gamma_{2,1}$};
\node[right] at (1.95,0.7) {$\Gamma_{2,2}$};

\draw[->] (1,0.25) -- (1.2,0.25);
\node[above] at (1.1,0.25) {$v$};

\end{tikzpicture}
\caption{A simple corridor bounded by two corner points} \label{fig1}
\end{center}
\end{figure}
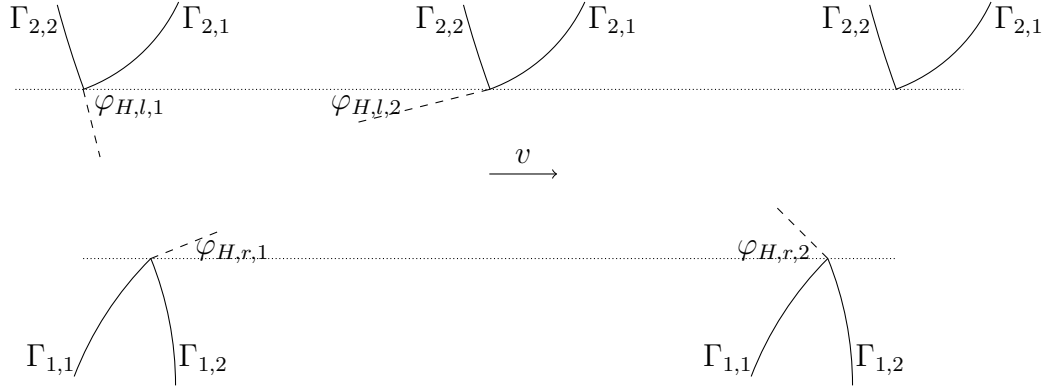

Now we define $A = \cup_{H \text{ corridors}} A_H$. With these notations, we are ready to introduce 
our assumptions
%We say that an admissible billiard table satisfies the {\it simple critical non-periodicity condition} if
\begin{itemize}
\item[(A1)] $\mathcal D$ is simple.
\item[(A2)] For any $(r,\varphi) \in A$, if $r$ corresponds to a corner point, then $|\varphi| \neq \pi/2$.
%\item[(A3)] Let $A' = \{ (r,\varphi) \in A: \text{$r$ corresponds to a corner point}\}$. Then for every $n \geq 1$, 
%$A' \cap F^{-n}(A') = \emptyset$. 
\end{itemize}

It seems likely that our
results remain true if we 
remove (A1) and (A2) but the proof becomes more complicated 
so we assume them for convenience. 

%replace (A1) - (A3) by
%\begin{itemize}
%\item[(A')] For any $H$ that is only bounded by corner points, none of the points $(r_{H,s,k},\varphi_{H,s,k}) \in A_H$ 
%is periodic.
%\end{itemize}
%However, for convenience we assume the slightly stronger assumptions (A1) - (A3).

%...

%Following \cite{CZ09}, given an unstable curve $W$, we denote by 
%$V_{m,\alpha}$  the connected components of $T^mW$. and write $W_{\alpha} = T^{-m} V_\alpha$. We say that the 
%$m$-step expansion holds if 
%$$
%\lim_{\delta \rightarrow 0} \sup_{W: |W|<\delta} \sum_{\alpha} \left( \frac{|W|}{|V_\alpha|}\right)^q
%\frac{|W_\alpha|}{|W|} <1.
%$$
%for some $q \leq 1$. By Jensen's inequality, the $m$-step expansion with some $q$ implies the 
%$m$-step expansion will all $q' \in (0,q)$.
%We note that the above limit is traditionally written as $\liminf$ however the sequence as $\delta \rightarrow 0$ is 
%non-increasing and bounded below and so the limit always exists.

\subsection{Definitions}
\label{sec:def}

We will denote by $C$ any constant only depending on $\mathcal D$, whose explicit value is irrelevant. In particular, the value of $C$ may change
from line to line. 

The billiard map $F$ is hyperbolic and ergodic. In particular, there exists uniformly transversal families of stable and unstable cones. Specifically, 
there are cones $\mathcal C^{u/s}_x$ for every $x \in \mathcal M$
so that $D_x F(\mathcal C^u_x) \subset F(\mathcal C^u_{F(x)}) \cup \{ 0 \}$
and 
$D_x F^{-1}(\mathcal C^s_x) \subset F(\mathcal C^s_{F^{-1}(x)}) \cup \{ 0 \}$.
Furthermore, there is a positive number $a$ so that for any $x \in \mathcal M$ 
and any
$(d r_1, d \varphi_1) \in \mathcal C^{u}_x $,
$(d r_2, d \varphi_2) \in \mathcal C^{s}_x $, 
we have $a \leq d \varphi_1/d r_1$ and $d\varphi_2/d r_2 \leq -a$. Furthermore, at least one of the following two inequalities hold:
$d\varphi_1/ dr_1 \leq a^{-1}$, $- a^{-1} \leq d\varphi_2/dr_2$
(see \cite[Section 9]{C99}. We note that all of these inequalities hold when there are no corner points.) 
Vectors in $\mathcal C^{u}_x $ ($\mathcal C^{s}_x $) are called unstable (stable).
It follows that the exists some number $\gamma >0$ so that at any point $x \in \mathcal M$, 
the angle between any stable and unstable vector
is bounded from below by $\gamma$ and no horizontal vector (that is $d\varphi = 0$) can
be in the stable/unstable cones. 
The standard way of defining the stable/unstable cones is 
as the $DF$/$DF^{-1}$ image of the cones $dr d \varphi \geq 0$.
We will briefly refer to these properties as 
transversality.

Hyperbolicity needs to be understood in the sense that for almost every point there is 
an unstable and a stable manifold through this point, however they can be arbitrarily short.
This is due to the singularities. 
The hyperbolicity is uniform in the sense that there are constants $C_\#$ and $\Lambda_* >1$ so that
for any $n \geq 1$,
for every unstable vector $u$,
\begin{equation}
\label{eq:hyper}
\| DF^n(u) \| \geq C_\# \Lambda_*^{n} \|u \|
\end{equation}
and for any stable vector $v$,
$$
\| DF^{-n}(v) \| \geq C_\# \Lambda_*^{n} \|v \|.
$$
%Furthermore, the stable and unstable cones are uniformly transversal.
%Following \cite{BSCh90, DST13}, 
Let us write
$$
S_0 = \cup_{i,j} \Gamma_{i,j} \times \{\pm \pi /2\}, \quad V_0 = \cup_{i} \cup_j \partial \Gamma_{i,j} \times
[\pi/2, \pi/2],
$$
that is $S_0$ is the set of grazing collisions and $V_0$ is the set of collisions at the corner points,
and $R_0 = S_0 \cup V_0$. Furthermore, let $R_{m,n} = \cup_{l=m}^n F^l R_0$. Then for any $n \geq 1$ (including $n = \infty$), the singularity set of $F^n$
is $R_{-n,0}$ and the singularity set of $F^{-n}$
is $R_{0,n}$.
Furthermore,  as usual, we introduce secondary (artificial) singularities
$$
\hat S_{\pm k} = \{ (r, \phi): \phi = \pm \pi/2 \mp k^{-2}\}
$$
for some $k \geq k_0$ to control distortion.
We denote 
$$\mathbb H_0
= {\{ r, \varphi \in \mathcal M: - \pi/2 - k_0^{-2} \leq \varphi \leq  \pi/2 - k_0^{-2}  \}}
$$
$$\mathbb H_k
= {\{ r, \varphi \in \mathcal M: \pi/2 - k^{-2} \leq \varphi \leq  \pi/2 - (k+1)^{-2}  \}}
$$
$$
\mathbb H_{-k}
= {\{ r, \varphi \in \mathcal M: -\pi/2 + (k+1)^{-2} \leq \varphi \leq - \pi/2 + k^{-2}  \}}.
$$ 
The extended set of singularities is
$$
R_0^{\mathbb H} = S_0^{\mathbb H} \cup V_0,  \text{ where } S_0^{\mathbb H} = S_0 \cup (\cup_{k \geq k_0} \hat S_{\pm k} )
$$
and likewise $R_{n,m}^{\mathbb H} = \cup_{l=m}^n F^l R_0^{\mathbb H}$. 
We say that a $C^2$ curve in $W \subset \mathcal M$ is unstable if at every point $x \in W$, the tangent line
$T_xW$ is in the unstable cone, and $W$ has a uniformly bounded curvature and is disjoint to 
$R_0$ (except possibly for its endpoints). 
We say that an unstable curve is homogeneous if it lies entirely in $\mathbb H_k$ for some $k$ ($k=0$ or $|k| \geq k_0$).
It is useful to think about unstable curves as smooth curves in the 
northeast-southwest direction on $\mathcal M$.

As in \cite[section 4]{ChD09}, 
we say that $\ell = (W, \rho)$ is a standard pair if $W$ is a homogeneous unstable curve and 
$\rho$ is a probability measure supported on $W$ that satisfies
$$
\left|
\log \frac{d \rho}{d Leb} (x) - \log \frac{d \rho}{d Leb} (y)  
\right| \leq C_0
\frac{|W(x,y)|}{|W|^{2/3}}.
$$
Note that there is some constant $C$ so that for any standard pair $\ell = (W, \rho)$ and for any $x, x' \in W$ 
\begin{equation}
\label{eq:densityosc}
e^{- C |W|^{1/3}} \leq \frac{\rho(x)}{\rho(x')} \leq e^{ C |W|^{1/3}}.
\end{equation}

The image of a standard pair is a weighted average of standard pairs. More precisely, if $\ell = (W, \rho)$ is a standard pair
and $\nu_{\ell}$ is the measure on $W$ with density $\rho$, then $F(W) = \cup_i W_i$, where $W_i$ are
homogeneous unstable curves and 
$F_*(\nu_{\ell}) = \sum_i c_i \nu_{\ell_i}$, where $\ell_i = (W_i, \rho_i)$ are standard pairs 
(see \cite[Proposition 4.9 ]{ChD09}).
We will also write $F_*({\ell}) = \sum_i c_i {\ell_i}$.

Substandard families are weighted averages of standard pairs
where the sum of the weights is $\leq 1$. That is, $\mathcal G = 
((W_{\alpha}, \rho_{\alpha})_{\alpha \in \mathfrak A}, \lambda)$ is a substandard family if 
$(W_{\alpha}, \rho_{\alpha})$'s are standard pairs and $\lambda$ is a subprobability 
measure on $\mathfrak A$. 
We assume that the $W_{\alpha}$'s are disjoint. Given a 
substandard family 
$\mathcal G$, it induces a measure $\nu_{\mathcal G}$ on $\mathcal M$ by
$$
\nu_{\mathcal G} (B) = \int_{\alpha \in \mathfrak A} \nu_{\alpha} (B \cap W_{\alpha}) d \lambda (\alpha)
\text{ for } B \subset \mathcal M \text{ Borel sets, }
$$
where $\nu_{\alpha}$ is the measure on $W_{\alpha}$ with density $\rho_{\alpha}$.
In case $\lambda_{\mathcal G}$ is a probability measure, we call $\mathcal G$ a standard family.
Now given a point $x \in W_{\alpha}$, denote by $r_{\mathcal G}(x)$ the distance between $x$ and 
the closest endpoint of $W_{\alpha}$ (measured along $W_{\alpha}$ with respect to arclength). 
We introduce the notion of
the $\mathcal Z_q$ function for $q \in (0,1]$ by
$$
\mathcal Z_q(\mathcal G) = \sup_{\varepsilon >0} 
\frac{\nu_{\mathcal G} (r_{\mathcal G} < \varepsilon) }{\varepsilon ^ q}.
$$
For example if $\mathcal G$ consists of only one standard pair $(W, \rho)$, then 
\begin{equation}
\label{eq:Zas}
\mathcal Z_q(\mathcal G) \sim 2^q/|W|^q
\end{equation}
as $|W| \rightarrow 0$.
A fundamental fact about the class of standard families is that they are preserved under iterations of the map $F$.

Given a homogeneous unstable curve $W$, its image $F^m(W)$ will consist of a collection of homogeneous
unstable curves $W_i$. For each $i$, let $\Lambda_i = \Lambda_{i,m}$ be the minimal expansion factor of
$F^m$ on $F^{-m} W_i$. We say that the $m$-step expansion holds if 
$$
\lim_{\delta \rightarrow 0} \sup_{W: |W|< \delta} \sum_i \frac{1}{\Lambda_{i,m}} < 1,
$$
where the supremum is taken over homogeneous unstable curves.
We note that the above limit is traditionally written as $\liminf$, however the sequence as $\delta \rightarrow 0$ is 
non-increasing and bounded below, so the limit always exists.

For given $n$, let $\xi^n$ be the partition of $\mathcal M \setminus R_{0,n}^{\mathbb H}$ into connected components.
Now the forward separation time of points $x, y \in \mathcal M$, denoted by $s_+(x,y)$, is defined as the smallest $n$ 
so that $x$ and $y$ belong to different partition elements of $\xi^n$. Likewise, we define $s_-(x,y)$ as the backward separation time. 
We say that $f: \mathcal M \rightarrow \mathbb R$
is dynamically H\"older if
there are constants $C = C(f)$ and $\vartheta = \vartheta (f) < 1$ so that
for any $x,y$ on the same unstable manifold $W^u$,
$$
|f(x) - f(y)| \leq C \vartheta^{s_+(x,y)}
$$
and for any $x,y$ on the same stable manifold $W^s$,
$$
|f(x) - f(y)| \leq C \vartheta^{s_-(x,y)}.
$$

We will write $a_n \sim b_n$ if $\lim_{n \rightarrow \infty} a_n / b_n = 1$ and $a_n \approx b_n$ if 
there are is a constant $C$ so that  $1/ C \leq a_n / b_n \leq C$ for all $n$.

\section{Results}
\label{sec:res}

Now we can state our results.
\begin{theorem}
\label{thm1}
For all type D1 billiard tables there is some $m_0$ so that the $m_0$-step expansion 
holds.
\end{theorem}

\begin{theorem}
\label{thm2}
For all type D2 billiard tables satisfying (A1) and (A2) 
and for all $q \in (0,1)$ we have the following.
There are 
constants $M\in \mathbb N$,
$\varkappa <1$ and $\delta_0 >0$ so that for any standard pair $\ell = (W, \rho)$
with $|W| < \delta_0$, $\mathcal Z_q(F^M_*(\ell)) < \varkappa \mathcal Z_q(\ell)$.
%There are constants $C_1$, $C_2$ and $\theta <1$ so that for all standard families 
%$\mathcal G$ with $\mathcal Z_q (\mathcal G) < \infty$, and for all $n \geq 1$, we have
%$$
%\mathcal Z_q (T^n \mathcal G) \leq C_1 \theta^n \mathcal Z_q (\mathcal G) + C_2
%$$
\end{theorem}

\begin{theorem}
\label{thm3}The set of billiard 
tables satisfying (A1) and (A2) is open and dense in $\bm D$.
\end{theorem}

\begin{proof}[Proof of Theorem \ref{thm0} assuming Theorems \ref{thm1}, \ref{thm2}]

In the case of type D1 billiard tables, the 
key estimate is the $m_0$-step expansion, provided by Theorem \ref{thm1}.
Once it is established, the exponential decay of correlations and the central
limit theorem
follow from the theory developed in \cite{C99} 
as it was also noted in case of type C billiards in \cite{DST13}. 

In the case when the assumptions (A1) and (A2) are satisfied,
the result follows from \cite{CZ09}. In that work, 
the exponential decay of correlation follows from some abstract assumptions denoted by (H1) - (H5). 
In our case, assumptions (H1)-(H4) are standard
as is usual for billiards (see e.g. \cite[section 9]{C99}). We do not know how to prove (H5) 
(see Remark \ref{rem:expansion} below)
however we have Theorem \ref{thm2} instead. In the proof of \cite{CZ09}, (H5) is only used to derive 
the growth lemma (see \cite[Lemma 3(a)]{CZ09}) which is standard once 
our Theorem \ref{thm2} holds. 
That being said, we can replace (H5) by Theorem \ref{thm2} and conclude 
the result of Theorem \ref{thm0}.
\end{proof}

\begin{remark}
\label{rem0}
Under the same assumptions as Theorm \ref{thm0}, several other results follow immediately from the abstract theory. 
Indeed, a "magnet" is constructed in \cite{CZ09} which implies the existence of a Young tower \cite{Y98} with exponential return times.
Thus the central limit theorem \cite{Y98},
large deviation principles \cite{RY08},
almost sure invariance principle, law of iterated logarithm \cite{C06, PhS75}, etc.
follow.
%local limit laws [14], almost sure invariance principles [10], and Berry–Esseen
%type theorems [12], etc.
\end{remark}

\begin{remark}
It is important to note that the test functions in the setup of Theorem \ref{thm0} and
Remark \ref{rem0} are assumed
to be bounded and dynamically H\"older. 
Important functions of interest are the free flight time
$\tau: \mathcal M \rightarrow \mathbb R$ and the displacement vector
$\kappa: \mathcal M \rightarrow \mathbb R^2$ defined as $\Pi_{\mathcal D}(F(x)) - \Pi_{\mathcal D}(x)$, lifted to the universal cover $\mathbb R^2$.
Both of these functions are dynamically H\"older but unbounded.
In particular, we do not claim the central limit theorem for 
the position of the billiard particle in the periodic extension of $\mathcal D$ from $\mathbb T^2$ to $\mathbb R^2$ (also 
known as Lorentz gas). In fact, we expect that the this position will converge to the normal distribution under the 
scaling $\sqrt{t \log t}$ (where $t$ is continuous time in case the flow is considered, 
and 
collision time in case of the map) as in \cite{SzV07}, but we do not study this question here.
\end{remark}

%%%%%%%%%%%%%%%%%%%%%%%%%%%%%%%%%%%%
\section{Proof of Theorem \ref{thm1}}
\label{sec:thm1}

The proof of Theorem \ref{thm1} is based on similar proofs for type C as in
\cite{DST13} 
and type B as in \cite{CM06}. As such, we only give detailed 
arguments
at places where our proof differs from these references and 
otherwise cite the necessary lemmas from \cite{CM06,DST13}.

First we review the structure of corridors and singularities, see \cite[Section 4.10]{CM06} for details.
Let us fix a regular billiard table. There are finitely many points,
$$A =\{ x_h = (r_h, \varphi_h), h=1,...,h_{\max}, \varphi_h = \pm \pi/2\}$$
that are periodic and whose trajectories bound the corridors.
The singularity structure of $F$ and $F^{-1}$ near $x_h$ is as follows. There are infinitely many singularity curves
accumulating at $x_h$. Specifically, there are connected components $D^+_{h,n}$ for $n \geq 1$ in 
a neighborhood of $x_h$ so that $F$ is smooth on $D^+_{h,n}$ and 
the trajectories of the points in $D^+_{h,n}$
pass by $n$ copies of the given scatterer before the next collision. Likewise,
$D^-_{h,n}$ is a set on which $F^{-1}$ is continuous, and $F(D^+_{h,n}) = D^-_{h',n}$ for some $h'$.
The size of $D^+_{h,n}$ is $\approx n^{-2}$ in the unstable direction and $\approx n^{-1/2}$ in the 
stable direction. Likewise, the size of $D^-_{h,n}$ is $\approx n^{-1/2}$ in the unstable direction and $\approx n^{-2}$ in the stable direction. Consequently, 
$D^-_{h,n}$ intersects with $\mathbb H_k$ if $|k| \geq C n^{1/4}$. 
The rate of expansion of $F$ on $D^+_{h,n} \cap F^{-1}(\mathbb H_k)$ is $ \approx n k^2$.

We start by the following

\begin{lemma}
\label{lemma:upperbdlen}
There is a constant $C$ so that for any unstable curve $W$ and for any homogeneous 
connected component 
$W' \subset F(W)$, we have
$$
|W'| \leq C |W|^{1/3}.
$$
\end{lemma}

\begin{proof}
Let $N$ be an arbitrary positive integer. If $W \not \subset \cup_h \cup_{n \geq N} 
D^+_{h,n}$, then the stronger bound 
$|W'| \leq C |W|^{1/2}$ holds with $C=C(N)$ by \cite[Exercise 4.50]{CM06}.
Let us assume now that $W \subset \cup_{n \geq N} D^+_{h,n}$. We can also assume that
there is some $n$ so that $W \subset D^+_{h,n}$. Indeed, for any connected component 
$W' \subset F(W)$, by connectedness, there is some $n= n_{W'}$ so that $F^{-1}(W') \subset D^+_{h,n}$.
Restricting $W$ to $D^+_{h,n_{W'}}$ will decrease the length of $W$ but not change the length of
$W'$. Thus it is enough to prove the lemma when $W$ is replaced by $W \cap D^+_{h,n_{W'}}$.

Since $W'$ is a homogeneous unstable curve, there is  some $k$ so that
$W' \subset \mathbb H_{\pm k}$. 
We can also assume that 
$F(W) = W'$. Indeed, if $W' \subset F(W)$, the proof is easier as 
$|F(W)| \geq |W'|$.
By the discussion right before the lemma, we have $|W'| / |W| \approx nk^2$. Now extend $W$ smoothly to a longer unstable 
curve $\bar W$ so that
$\bar W' = F(\bar W)$ is maximal with $\bar W' \subset \mathbb H_{\pm k}$, i.e. the two endpoints
of $\bar W'$ belong to $\partial \mathbb H_{\pm k}$. Then 
$|\bar W'| / |\bar W| \approx |W'| / |W| \approx nk^2$. Recalling that $|k| > n^{1/4}$,
we find $|\bar W'| \leq C |\bar W| k^6$.
Next, 
by transversality, we have
$|\bar W'| \approx |k|^{-3}$ and so $|\bar W'| \leq C |\bar W|^{1/3}.$
We conclude 
$$|W'| \leq C |W| \frac{|\bar W'|}{|\bar W|}
\leq C|W| |\bar W|^{-2/3} \leq C | W|^{1/3}.
$$

\end{proof}

\begin{remark}
In the case of finite horizon if we do {\it not} require $W'$ to be homogeneous, we have
$|W'| \leq |W|^{1/2}$
by \cite[Exercise 4.50]{CM06}.
Similarly to the above proof, one can show that in case of finite horizon and if
$W'$ is required to be homogeneous, then $|W'| \leq |W|^{3/5}$ holds. Furthermore,
in the case of infinite horizon if $W'$ is not required to be homogeneous, then the weaker bound
$|W'| \leq |W|^{1/4}$ holds. Since we will not use these bounds, the proofs are omitted.
\end{remark}

We will also need an estimate on the growth of the free flight function along orbits, which is provided by the next lemma.

\begin{lemma}
\label{lem:flightgrowth}
There are constants $\Cl{c:conseclongflight}, t_0,t_1$ only depending on $\mathcal D$ so that for any point $x \in \mathcal M$
with $\tau(x) > t_1$, there are two possibilities:
\begin{itemize}
\item
either $\tau(F(x)) \in 
[\Cr{c:conseclongflight}^{-1}\sqrt{\tau(x)}, \Cr{c:conseclongflight} (\tau(x))^2]$
\item or 
$\tau(F(x)) < t_0$ in which case
$\tau(F^2(x)) \in [\Cr{c:conseclongflight}^{-1}\sqrt{\tau(x)}, \Cr{c:conseclongflight} (\tau(x))^2]$.
\end{itemize}
\end{lemma}

We don't give a formal proof of Lemma \ref{lem:flightgrowth},
as the first case was proved in \cite[Proposition 9]{SzV07}
and the second case is similar. Instead, we give an explanation.
A long flight needs to happen in a corridor, e.g. in the northeast direction in a horizontal corridor
with an angle $\alpha \approx 1/\tau(x)$. Now let $P, P' \in \partial \mathcal D$ be
two consecutive points on the boundary of the corridor (see Figure \ref{figtype1})
with $P \in \partial \mathcal B$, $P' \in \partial \mathcal B'$
(here $\mathcal B = \mathcal B_i + m$ and 
$\mathcal B' = \mathcal B_i + m'$
for some $i=1,...,d$ and 
$m,m' \in \mathbb Z^2$).
The free flight of $x$
will intersect the line segment $(P, P')$. Let $Q,Q'$ be two more points on
this line segment with $P<Q<Q'<P'$,
where $X<Y$ means that $Y$ is to the right of $X$. 
Here $Q'$ is defined so that the line segment
with angle $\alpha$ through $Q'$ is tangent to $\mathcal B$ and $Q$ is defined by
the property that a billiard particle travelling 
with angle $\alpha$ through $Q$ first collides with $\mathcal B$ and then experiences a grazing
collision on $\mathcal B'$. 

Let $R$ denote the point where the free flight, starting from $x$, crosses $(P,P')$.
If $P\leq R<Q$, then 
$\tau(F(x)) \in 
[\tau(x), \Cr{c:conseclongflight} (\tau(x))^2]$.
If $Q\leq R \leq Q'$, then $\tau(F(x)) < \| m'-m \|$ and 
$\tau(F^2(x)) \in [\Cr{c:conseclongflight}^{-1}\sqrt{\tau(x)}, \Cr{c:conseclongflight} (\tau(x))^2]$. Finally, if $Q' < R < P'$, then
$\tau(F(x)) \in  [\Cr{c:conseclongflight}^{-1}\sqrt{\tau(x)},
\tau(x))$. Also note that the last case is the typical one in the sense that $\|P - Q'\| \approx 1/\tau(x)$.

Figure \ref{figtype1} represents the two cases
of Lemma \ref{lem:flightgrowth}. 
The dashed line is a reference trajectory through $Q'$.
The solid line is a trajectory with $Q' < R < P'$ and the dash-dotted line
is a trajectory with $Q<R<Q'$.
(Note that for the sake of legibility, this image doesn't exactly reflect how small the angles would be.)
\begin{center}
\begin{figure}
    \includegraphics[scale = 0.7]{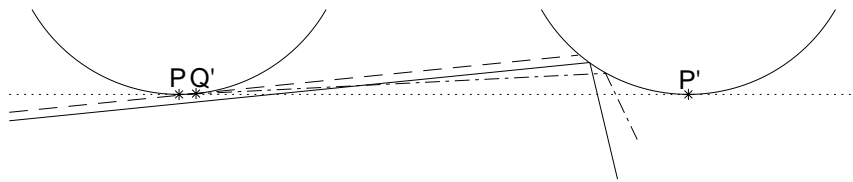}
\caption{Collisions after a long flight}\label{figtype1}
\end{figure}
\end{center}

Now fix some $\varepsilon_0 >0$, $n_0$ and $\bar \tau$ so that the following is true:
For any point $x=(r ,\varphi) \in \mathcal M$ 
that has a free flight longer than $\bar \tau$, there is some $h = 1,...,h_{\max}$ so that 
$d(x,x_h) < \varepsilon_0$ and $x \in D^+_{h, n}$ for some $n \geq n_0$. 
Furthermore, the trajectory of $x$ under the billiard flow until the
next collision avoids the $\varepsilon_0$ neighborhood of all corner points in $\mathbb T^2$.

Let us write $\mathcal M_{\mathfrak n} = 
\mathcal M \setminus \cup_{h = 1}^{h_{\max}}
\cup_{n \geq \mathfrak n} D^+_{h, n}
$ and $\mathcal M_{m,\mathfrak n} = \cap_{l  = 0}^{ m} F^{-l} \mathcal M_{\mathfrak n}$.
%We remark that a simple geometric argument shows that there is some $C$
%so that $\tau(F(x)) < C\tau(x)^2$
%and so $\tau(F^m(x)) < \tilde C^m \tau(x)^{2m}$
Note that for $m$ fixed and for $\mathfrak n$ large enough, we have
\begin{equation}
\label{eq:longflightiterated}
\mathcal M_{(\hat c \mathfrak n)^{1/(2^m)}} \subset
\mathcal M_{m,\mathfrak n} \subset \mathcal M_{\mathfrak n}.
\end{equation}
Indeed, the first inequality follows from
Lemma \ref{lem:flightgrowth} and the second one is trivial.

Following \cite{DST13}, we introduce the following definitions.
Let $W$ be a homogenenous unstable curve and $W_{i,m}$ the homogeneous components
(H-components)
of $F^m(W)$. We say that $W_{i,m}$ is $m$-regular if 
$F^{-l}(W_{l,m}) \in \mathbb H_0$
%and $W_{i,m}$ is extended $m$-regular if 
%$F^{-l}(W_{l,m}) \in
%\hat{\mathbb H}_0 := 
% \mathbb H_0 \cup \mathbb H_{k_0} \cup \mathbb H_{-k_0}$
for all $0 \leq l < m$. 
If $W_{i,m}$ is not $m$-regular, it is called $m$-nearly grazing. 

Given some $\mathfrak n$, $m$ and $z \in \mathcal M_{m,\mathfrak n}$, we define 
$
\mathcal K_{m, \mathfrak n}^{reg} (z)
$
 as the number of connected components of 
$\mathcal M' \subset \mathcal M_{\mathfrak n} \setminus R_{0,m}^{\mathbb H}$
so that the closure of $\mathcal M'$ contains $z$ and some (and consequently all) points 
$x \in \mathcal M'$ satisfy
\begin{equation}
\label{eq:hath}
F^l(z) \in \hat{\mathbb H}_0 \text{ for all } l = 0,...,m,
\end{equation}
where
$$
\hat{\mathbb H}_0 := 
\mathbb H_0 \cup \mathbb H_{k_0} \cup \mathbb H_{-k_0}.
$$
Let 
$$
\mathcal K_{m}^{reg} = 
\sup_{\mathfrak n \geq 1}
\sup_{z \in \mathcal M_{m,\mathfrak n}}\mathcal K_{m, \mathfrak n}^{reg} (z)
$$

Now we have
\begin{lemma}
\label{lem:comp}
There is some $\Xi$ depending only on $\mathcal D$ (and in particular not
depending on $k_0$) so that 
$$
\mathcal K_{m}^{reg} 
=\sup_{\mathfrak n \geq 1}
\sup_{z \in \mathcal M_{\mathfrak n}}\mathcal K_{m, \mathfrak n}^{reg} (z)
\leq \Xi (m+1)
$$
\end{lemma}

\begin{proof}
By \eqref{eq:longflightiterated}, the first equality follows. The inequality 
follows from \cite[Lemma 3.5]{DST13}. Although that lemma is only proved
in the case of finite horizon, the proof applies in our case as well. 
Let us fix $\mathfrak n_0 = \max \{ n_0, C k_0^4 \}$. Then the proof of \cite{DST13} 
implies
$\sup_{z \in \mathcal M_{\mathfrak n_0}}\mathcal K_{m, \mathfrak n_0}^{reg} (z)
\leq \Xi (m+1)$. We just need to replace $\tau_{\max}$
by $\bar \tau$ in Lemma 3.6; in particular $\Xi = 4 \bar \tau/ \tau_{*} +6$ works,
where $\tau_{*}$ is the length of the minimal free flight between two improper collisions 
(a geometric constant only depending on $\mathcal D$).
Indeed, whenever $\tau(x) > \bar \tau$ (this can happen if $\mathfrak n > n_0$), the trajectory
up to the next collision avoids the $\varepsilon_0$ neighborhood of the corner points by the choice of 
$n_0$, so Lemma 3.6 remains valid. 
Now if $\mathfrak n > \mathfrak n_0$, then by the choice of $\mathfrak n_0$
and by Lemma \ref{lem:flightgrowth}, all points 
$z \in \mathcal M_{m,\mathfrak n} \setminus \mathcal M_{m,\mathfrak n_0}$ satisfy
$F(z) \in {\mathbb H}_k$ for some $k$ with $|k| > k_0 $. Thus
$\mathcal K_{m, \mathfrak n}^{reg} (z) = 0$.
%so
%$
%\mathcal K_{m, \mathfrak n}^{reg} (z)
%= 
%\mathcal K_{m-1, \mathfrak n}^{reg} (F(z))
%$. Applying this argument inductively until the first time $l$ so that $F^l(z) \in \mathcal M_{n_0}$, 
%we obtain the lemma.

\end{proof}

Let $k_0$ be fixed and let $K_m^{reg}(W)$ denote the number of $m$-regular
H-components of $F^m(W)$. Then we have
\begin{lemma}
\label{lem:comp2}
There exists some $m_0$ only depending on $\mathcal D$ so that for any $k_0$,
$$
\lim_{\delta \rightarrow 0} \sup_{W: |W| < \delta} K_{m_0}^{reg}(W) 
< \frac13 C_\# \Lambda_*^{m_0}
$$
where $C_\#$ and $\Lambda_*$ are defined by 
\eqref{eq:hyper}.
\end{lemma}
 
\begin{proof}
This lemma is proved as \cite[Lemma 2.12]{DST13}. 
We fix $m_0$ so that $\Xi (m_0 + 1) < \frac13 C_\# \Lambda_*^{m_0}$.
Now since $k_0$ is fixed, we can choose
$n_1 > \max \{ n_0, k_0^4 \}$. Then as in Lemma \ref{lem:comp}, if 
$z \notin \mathcal M_{m_0,n_1}$, then $F(z) \in \mathbb H_k$ for $|k| > k_0$, so
the H-component $W' \subset F^{m_0}(W)$ containing $F^{m_0}(z)$ must be $m_0$-nearly
grazing. 
%If there was only one point $z$ so that $\mathcal K_{m_0, n_1}^{reg}(z) >2$, then by
%transversality and by the fact that we used $\hat{\mathbb H}_0$ in \eqref{eq:hath},
%we would conclude that  
%$K_{m_0}^{reg}(W) \leq \mathcal K_m^{reg}$ and so the lemma would follow.
Once $n_1$ is fixed, there are only finitely many points $\mathcal Z = \{ z_1,...,z_Z \}$ where 
$\mathcal K_{m_0, n_1}^{reg}(z) > 2$.
By choosing $\delta = \delta(n_1)$ small, we can assume that our curve $W$ is only close to one of these
points and so by transversality and by the fact that we used $\hat{\mathbb H}_0$ in \eqref{eq:hath},
we find $K_{m_0}^{reg}(W) \leq \mathcal K_{m_0}^{reg}$, which by Lemma
\ref{lem:comp}
completes the proof.
%and so the previous argument can be generalized by choosing $\delta$ sufficiently small 
%(depending on the set $\mathcal Z$). 
%Thus $K_{m_0}^{reg}(W)$ is bounded by the number of pieces
%$W$ is cut by the set $R_{0, m_0} \cup \cup_{l=0}^{m_0} F^l \hat S_{\pm k_0}$
%Since we used $\hat{\mathbb H}_0$ in \eqref{eq:hath}, in order to count
%$K_{m_0}^{reg}(W)$ it is sufficient 
\end{proof}

Next, we bound the contribution of nearly grazing components for $m=1$.

\begin{lemma}
\label{lem:grazing}
For any $\varepsilon >0$ there exists $k_0$ so that 
$$
\lim_{\delta \rightarrow 0} \sup_{W: |W|< \delta} \sum^*_i \frac{1}{\Lambda_{i,1}} < \varepsilon,
$$
where $ \sum^*$ means that the sum is restricted to nearly grazing H-components $W_{i,1}$.
\end{lemma}

\begin{proof}
This lemma is analogous to \cite[Lemma 2.13]{DST13} but the proof differs substantially
as the free flight now is unbounded. However, we can use \cite[Remark 5.59]{CM06}.

Let us write $W_1 = W \setminus \mathcal M_{n_0}$ and $W_2 = W \cap \mathcal M_{n_0}$.
Let $S^{*j}$ be the sum corresponding to the images of $W_j$ for $j=1,2$.
As before, for any $k_0$ there exists some $\delta>0$ so that if $|W| < \delta$, then 
$F(W_1)$ have at most $ L = \bar \tau/ \tau_{\min} + 2$ connected components. 
Each of these components could be further cut by secondary singularities, so
$S^{*1} \leq L \sum_{k \geq k_0} Ck^{-2}\leq CL k_0^{-1}$ which is less then $\varepsilon /2$
assuming that $k_0 > 2 CL/\varepsilon$.
For any fixed $k_0$ and $n_0$, we can make $S^{*2} \leq \varepsilon /2$
by further reducing $\delta$ if needed, exactly as in \cite[Remark 5.59]{CM06}.
\end{proof}

Now Theorem \ref{thm1} follows from Lemmas 
\ref{lem:comp},
\ref{lem:comp2} and 
\ref{lem:grazing}
as in \cite{DST13}. Note that there is a typo at the middle of page 1231 in \cite{DST13}
as one only has 
\begin{equation}
\label{eq:subadd}
\mathcal L_{n+m}(W) \leq \sum_i 
\frac{1}{\Lambda_{i,n}}\mathcal L_m(W_{i,n}),
\end{equation}
where 
\begin{equation}
\label{eq:sumexp}
\mathcal L_n(W) = \sum_i \frac{1}{\Lambda_{i,n}}
\end{equation}
 (and the equation in \eqref{eq:subadd} may not hold), but this is enough
since Theorem \ref{thm1} only gives an upper bound.

%In fact we proved the following slightly stronger version of Theorem \ref{thm1}.

%%%%%%%%%%%%%%%%%%%%%%%%%%%%%%%%%%%%%%%%%%
\section{Proof of Theorem \ref{thm2}}
\label{sec:thm2}

By (A1) and (A2),
we can group the corridors into three categories: 
$H \in \mathcal H_1$ if $H$ is bounded by two regular points,
$H \in \mathcal H_2$ if $H$ is bounded by a regular point and a corner point and
$H \in \mathcal H_3$ if $H$ is bounded by two corner points.
We say that the corridor is of type 
$1,2,3$, respectively. We also say that a boundary point $x \in A_H$ with $H \in \mathcal H_j$ is of type $j$. 

The reason we assume (A1) and (A2) is to guarantee that there are only 3 types here. Without these assumptions, 
there would be many more types (see \cite{BSCh90} for a description of all
types in general).

Let us fix an enumeration of the set $\cup_H A_H$ as $\{ x_1,...,x_{h_{\max}}\}$. Note that it is possible
that $\Pi_{\mathcal D}(x_h) = \Pi_{\mathcal D}(x_{h'})$ for some $h \in H$, $h' \in H'$, $H \neq H'$ in case
$\Pi_{\mathcal D}(x_h)$ is a corner point. 
%This is not an issue as  $x_h \neq x_{h'}$.

Recall that in case of type 1 corridors
$$A_H = \{ (r_{H,r}, \pi/2), (r_{H,r}, -\pi/2), (r_{H,l}, \pi/2), (r_{H,l}, -\pi/2)\}.$$ 
Also recall the notation introduced in Section \ref{sec:thm1}:
for any point $x_h \in A_H$ for some $H \in \mathcal H_1$, we denote by 
$D^{\pm} (h,n)$ the domains where $F^{\pm}$ is smooth in a neighborhood of $x_h$. 
We also note that $F(D^+(h,n)) = D^-(h',n)$
with
$x_h = (r_{H,l/r}, \pm \pi/2)$ and $x_{h'} = (r_{H,r/l}, \mp \pi/2)$.

If $H \in \mathcal H_3$, then $A_H$ is of the form \eqref{eq:horizonbd}. To emphasize the difference from the previous case,
we will denote by $E^+(h,n)$ the set of points $x \in \mathcal M$ in a vicinity of $x_h$ so that the free flight of $x$ passes by $n$
copies of the scatterer before colliding on the other side of the corridor whenever 
$x_h$ is of type 3. A simple geometric
argument shows that $E^+(h,n)$ is of size $\approx n^{-2}$ in the unstable direction and $\approx n^{-1}$ in the stable direction
(see \cite[Section 4]{BSCh90}, and Figure \ref{fig2}). 
The image of $E^+(h,n)$ is in a small neighborhood of a point $(q,-v)$, where
$(q,v) \in A_H$.

Type 2 corridors will require special consideration. Let $x_h = (r_h, \varphi_h) \in A_H$ with $H \in \mathcal H_2$. If $r_h$ corresponds to a 
regular point, then we define $D^{\pm}(h,n)$ with the same asymptotic size as in case $\mathcal H_1$ and
if $r_h$ corresponds to a corner point, then we define $E^+(h,n)$ with the same asymptotic size as in case $\mathcal H_3$.
\begin{center}
\begin{figure}
\includegraphics[scale = 0.5]{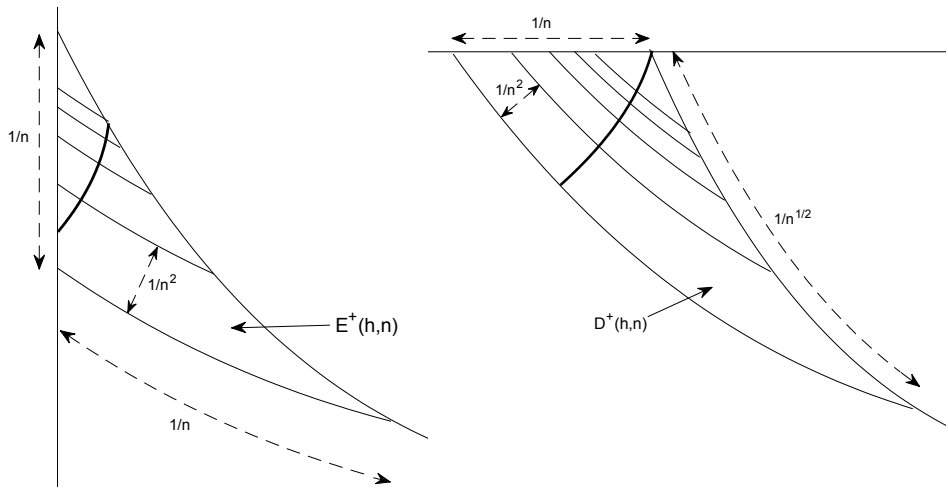}
\caption{Singularity structure near type 3 and type 1 boundary points. 
An unstable curve is indicated with bold on both panels.} \label{fig2}
\end{figure}
\end{center}
Recall that 
\begin{equation}
\label{def:A'}
A = \cup_H A_H \text{ and } A' = \{ x_h \in A: \Pi_{\mathcal D}x_h \text{ is a corner point }\} 
\end{equation}
Let us write 
\begin{equation}
\label{eq:defE}
\mathcal E(h,n) = \mathcal E(x_h,n)
= \cup_{N \geq n} E^+(h,N), \quad 
\mathcal E_n = \cup_{ h: x_h \in A' } \mathcal E(h,n),
\end{equation} 
\begin{align}
&\mathcal B(h,n) = \mathcal B(x_h,n)
= \cup_{N \geq n} D^+(h,N), \nonumber \\
&
 \label{eq:defB}
\mathcal B_n = \cup_{ h: x_h \text{regular boundary point of a type 2 corridor}} \mathcal B(h,n).
\end{align}
and finally
$$
\mathcal E_{n_1} \mathcal B_{n_2} = \mathcal E_{n_1} \cup \mathcal B_{n_2}
$$

That is, %$\mathcal E_n$ is the set of points close to a corner point that experience a long free flight and 
$\mathcal E_{n_1} \mathcal B_{n_2}$ is the set of points that experience a long free flight in type 2 or type 3 corridors.
This set is the disjoint union of $\mathcal E_{n_1}$ and $\mathcal B_{n_2}$,
where $\mathcal E_{n_1}$ is contained in a neighborhood of $A'$ and $\mathcal B_{n_2}$
is contained in a neighborhood of $A \setminus A'$. Note that when there are no type 2 corridors
we have 
$\mathcal B_{n_2} = \emptyset$. In this case, the forthcoming proof could be 
simplified substantially.

Let us fix a large integer $\Cl[Ns]{N:N_0}$ so that the sets $D^+(h,n)$ and $E^+(h',n')$
are disjoint whenever $n, n' \geq \Cr{N:N_0}$, $h \neq h'$.
We start with a geometric lemma.

\begin{lemma}
\label{lem:geomcornercorr}
There is a constant $\Cl{c:star}$ so that %the following is true.
for every unstable curve 
$W$ with $W \subset \mathcal E_{\Cr{N:N_0}}$,
%$length(W) < \delta_*$ and 
%$W \cap \mathcal E_{N_*} \neq \emptyset$ 
there is some $T=T(W)$
so that for all $x \in W$,
%with $\tau(x) > N_*$,  
$T \leq \tau(x) \leq \Cr{c:star} T$ holds.
\end{lemma}

\begin{proof}
%We can choose $\delta_*$ and $N_*$ so as any unstable curve
%curve $W$ with $length(W) < \delta_*$ can only intersect a small
%neighborhood of one point in $A'$. That is, if 
%$W \cap E^+(h,n) 
%\neq \emptyset$ for some $h$ with $x_h \in A'$ and
%$n \geq N_*$, then for all
%$h' \neq h$, $W \cap (\cup_{n \geq N_*} E^+(h,n) )= \emptyset$.
%Fix such a $\delta_*$. We may further increase $N_*$, the above
%condition will still hold. Now f
For every $h$ with $x_h = (r_h, \varphi_h) \in A'$, we
will show that the desired $\Cr{c:star}$ exists for unstable
curves in $\mathcal E(x_h, {\Cr{N:N_0}})$. 
This is sufficient as there are finitely
many corridors and we can take the biggest $\Cr{c:star}$.
Let us assume that $r_h$ is the left endpoint of the corresponding
boundary curve $\Gamma$ (the other case is similar).  

Without loss of generality we can assume that the endpoint of $W$
is a point $(r_h, \varphi_0)$. Indeed, if the curve does not stretch
all the way to the left boundary of $\mathcal M$, we can smoothly
extend it to the "southwest". 
A key observation is that $\varphi_0 < \varphi_h$. Indeed,
if $\varphi_0 \geq \varphi_h$, then since unstable curves are in
the "northeast direction", $W$ could not intersect 
$\mathcal E_{\Cr{N:N_0}}$. Now we choose 
$T=T(W) = \tau(r_h, \varphi_0)$. 
We can assume that the other edpoint
of $W$ is
$(r',\varphi') \in \partial \mathcal E_{\Cr{N:N_0}}$ so that the next collision
after leaving $(r',\varphi')$ is at the same corner point 
(and so is improper). Strictly speaking, $(r',\varphi')$ is not contained
in $W$ as $(r',\varphi') \notin \mathcal E_{\Cr{N:N_0}}$, that is,
$W$ is a curve that does not contain one of its endpoints.
Indeed, if $W$ does not fully
cross $\mathcal E_{\Cr{N:N_0}}$,
we can extend it smoothly to the northeast. 
By transverality, the triangle with vertices $(r_h, \varphi_h)$,
$(r_h, \varphi_0)$ and $(r',\varphi')$ has angles that are bounded
away from zero (specifically by an angle $\gamma$ as discussed
in Section \ref{sec:def}). 
Consequently, the distance between $(r_h, \varphi_h)$,
and $(r',\varphi')$ is bounded below by a universal constant
times $\varphi_h - \varphi_0$.
We also know that $\varphi_h - \varphi_0 \approx 1/T$,
whence
$\tau(r',\varphi') < \Cr{c:star} T$ for a geometric constant $\Cr{c:star}$. The lemma follows.
\end{proof}

Figure \ref{fig2} shows unstable curves (indicated with bold)
with long flight in corridors of type 1 and 3. 
In case of type 1 corridors, the free flight function restricted to an ustable curve
$W$ may be unbounded (see the right side of figure), but 
any long flight is necessarily followed by a nearly grazing collision.
This nearly grazing collision makes it easier to control the sum of expansion factors, 
as in the proof of 
Theorem \ref{thm1}. 
As seen 
on the left side of the figure and proven by Lemma \ref{lem:geomcornercorr}, 
for any unstable curve $W$, the free flight function restricted to $W$
is bounded in type 3 corridors (with a bound depending on $W$). We will leverage this fact
in Lemma \ref{lem:expansion0} to show that such a flight can only
increases the $\mathcal Z_1$ function
of a standard family by a bounded factor. 

Most work is required in case of type 2 corridors.
In this case, given an unstable curve near the regular boundary point
(as on the right panel of Figure \ref{fig2}) the free flight is unbounded
and after the collision, the expansion is {\it not} large. In particular, the sum in
\eqref{eq:sumexp} is infinite. To overcome
this difficulty, we prove in Lemma \ref{lem:expansion} that the $\mathcal Z_q$ function 
remains finite for any $q<1$. Then, we show that
multiple visits into corridors of type
2 or 3 in a short succession are not possible (Lemma \ref{lem:nonper}).
 
\begin{lemma}
\label{lem:expansion0}
There is some integer $\Cl[Ns]{n:type3} \geq \Cr{N:N_0}$ and a constant
$\Cl{C:type3factor}$ so that for every 
%$h$ with $x_h$ type $2$ or $3$ and for every
standard pair $\ell = (W, \rho)$ with 
$W \subset \mathcal E_{\Cr{n:type3}}$,
%or $W \subset  \cup_{n \geq N_0'} D^+_{h,n}$, we have
\begin{equation}
\label{eq:type3longest}
\mathcal Z_1(F_* (\ell)) \leq \Cr{C:type3factor} \mathcal Z_1 (\ell).
\end{equation}
\end{lemma}

\begin{proof}
Let $x_h = (r_h, \varphi_h)$ and $\ell = (W, \rho)$ be 
such that $W \subset \mathcal E(x_h, \Cr{n:type3})$ for some $\Cr{n:type3}$
to be specified later. Assume first
that $x_h$ is a type 3 boundary point.
Let us write $V_n = W \cap  E^+({x_h,n})$ 
and $W_n = F(V_n)$ for $n \geq \Cr{n:type3}$.
Next we claim that
there is a constant $\beta >0$ only depending on $\mathcal D$ such that
for all $(r, \varphi) \in W$ and with the notation $F(r,\varphi) = (r', \varphi')$, we have $\cos \varphi' \geq \beta$. 
Indeed, let $\beta =\alpha/2$, where $\alpha$ is the minimal angle between the half tangents of the boundary
points bounding the corridors and the directions of the corresponding corridor $v_H$ 
($\alpha >0$ by assumption (A2)). Then by choosing $\Cr{n:type3}$
sufficiently large, we can guarantee that the angle between
the line segment emanating from $(r,\varphi)$ and
$v_H$
is less than $\alpha/2$, which implies the claim.
Without loss of generality we can assume that $k_0 > \beta^{-1}$ and
so for all $n \geq \Cr{n:type3}$,
$W_n$ is a homogeneous unstable curve and the expansion of $F$ on $V_n$ is $\approx n$.
Furthermore, a simple geometric argument shows $|W_n| \approx n^{-1}$
whence $|V_n| \approx n^{-2}$. 
By Lemma \ref{lem:geomcornercorr}, there is some $n_{W}$
so that $F(W) = \cup_{n = n_W}^{\Cr{c:star}n_W} W_n$. We conclude
\begin{equation}
\label{eq:type3geomcomp}
\mathcal L_1(W) = 
\sum_{n = n_W}^{\Cr{c:star}n_W} \frac{1}{\Lambda_{n,1}} \leq
C \sum_{n = n_W}^{\Cr{c:star}n_W} \frac{1}{n} \leq 2 C \ln \Cr{c:star}.
\end{equation}
We obtained the variant of \eqref{eq:type3longest}, where
$\rho$ is constant. Generalizing it to all
admissible densities $\rho$ is standard only using
\eqref{eq:densityosc} and \eqref{eq:Zas} and so we omit the proof
(see e.g. \cite{CZ09}). 

In case $x_h$ is a type 2 boundary corner point, we write $V_n = W \cap  E^+({x_h,n})$ 
and $W_n = F(V_n)$ as before. Now $W_n$ is not a homogeneous curve as
it is further cut into infinitely many pieces by secondary singularities. However, 
on any of these pieces, the expansion factor is large and so $\mathcal L_1(W)$ is 
bounded as in \cite[Remark 5.59]{CM06}. The lemma follows.
\end{proof}

\begin{lemma}
\label{lem:expansion}
There is some integer $\Cl[Ns]{n:type2} \geq \Cr{n:type3}$ so that
for every $q < 1$ there is a constant
$\Cl{C:type2factor}$ so that for every 
%$h$ with $x_h$ type $2$ or $3$ and for every
standard pair $\ell = (W, \rho)$ with 
$W \subset \mathcal B_{\Cr{n:type2}}$,
%or $W \subset  \cup_{n \geq N_0'} D^+_{h,n}$, we have
$$
\mathcal Z_q(F_* \ell) \leq \Cr{C:type2factor} \mathcal Z_q (\ell).
$$
\end{lemma}

\begin{proof}
Let $x_h = (r_h, \pm \pi/2)$ be a regular boundary point of a type 2 corridor and let
$\ell = (W, \rho)$ be 
such that $W \subset \mathcal B(x_h, \Cr{n:type2})$ for some $\Cr{n:type2}$ 
to be specified later.
As in the case of type 3 corridors (cf. the proof of Lemma \ref{lem:expansion0}), we find that
$V_n = W \cap  D^+_{h,n}$, 
$W_n = F(V_n)$, the expansion of $F$ on $V_n$ is
$\approx n$, $|W_n| \approx n^{-1}$, and
$|V_n| \approx n^{-2}$.
The main difference from the case of type 3 corridors is that now Lemma \ref{lem:geomcornercorr}
fails to hold. Indeed, the curve $W$ may be cut into {\it infinitely
many} pieces (see the right panel of Figure \ref{fig2}) and so the sum in \eqref{eq:type3geomcomp} diverges.
We are going to prove that 
\begin{equation}
\label{eq:flightsum}
 \sum_{n \geq \Cr{n:type2}} \left( \frac{|W|}{|W_n|}\right)^q
\frac{|V_n|}{|W|} \leq  \Cr{C:type2factor}.
\end{equation}
First note that \eqref{eq:flightsum} implies the lemma when $\rho$ is constant.
The general case can be proven
using \eqref{eq:densityosc} and \eqref{eq:Zas}. Thus it remains to verify \eqref{eq:flightsum}. To prove
\eqref{eq:flightsum}, we distinguish three cases.

First assume that $W$ is cut into infinitely many pieces, that is, $V_n \neq \emptyset$
for infinitely many $n$'s. Then $x_h$ is necessarily an endpoint of $W$. Let $M$ be the smallest
integer $n$ so that $W$ fully crosses $D^+_{h,n}$. Then we have $|W| \approx M^{-1}$ and so 
$$
 \sum_{n \geq \Cr{n:type2}} \left( \frac{|W|}{|W_n|}\right)^q
\frac{|V_n|}{|W|} \leq C |W|^{q-1} \sum_{n \geq M-1} n^{-2 + q} 
\leq C |W|^{q-1} \sum_{n \geq  \bar C |W|^{-1}} n^{-2 + q} \leq  \Cr{C:type2factor}.
$$

Next, assume that $W \subset D^+_{h,M}$ for some $M$. Then the left hand side of 
\eqref{eq:flightsum} is $(|W| / |W_M|)^q \approx M^{-q}$, which is also bounded.
 
Finally, assume that there are positive integers $M_1 \leq M_2$ so that
$V_n \neq \emptyset$ if and only if $M_1 \leq n \leq M_2 $. The contribution of 
$n=M_1$ and $n=M_2$ is bounded as in the second case. Thus it suffices to bound the contribution
of $n=M_1+1,...,M_2-1$, that is the set of $n$'s so that $W$ fully crosses $D^+_{h,n}$.
To simplify the notation, we replace $M_1 + 1$ by $M_1$ and $M_2 -1 $ by $M_2$. We obtain
$$
\sum_{n = M_1}^{M_2} \left( \frac{|W|}{|W_n|}\right)^q
\frac{|V_n|}{|W|} \leq C |W|^{q-1} \sum_{n = M_1}^{M_2} n^{-2 + q} 
\leq C (M_1^{-1} - M_2^{-1} )^{q-1} (M_1^{q-1} - M_2^{q-1} ).
$$
Writing $a = M_2 / M_1 -1$, we find that the above display is bounded by
$$
C(M_2^{-1} a)^{q-1} M_2^{q-1} \left[ \left( \frac{M_1}{M_2} \right)^{q-1} -1\right]
= Ca^{q-1} [(1+a)^{1-q} -1] =: Cf(a).
$$ 
Now $f$ is a continuous function on $\mathbb R_+$ with $\lim_{a \rightarrow \infty} f(a) = 1$ and
$\lim_{a \rightarrow 0} f(a) = 0$ (in fact $f(a) \sim (1-q)a^q$ as $a \rightarrow 0$). Consequently,
$f$ is bounded and so \eqref{eq:flightsum} follows.
\end{proof}

\begin{remark}
\label{rem:expansion}
The proof of Theorem \ref{thm2} could be simplified if we knew that 
the constant $\Cr{C:type2factor}$ given by Lemma \ref{lem:expansion} is less than $1$.
Although the 1-step expansion as required by (H5) of 
\cite{CZ09} would not follow, 
as it does not even hold in type C billiards, 
at least long flights in a short succession would not cause a trouble.
Unfortunately we do not know whether $\Cr{C:type2factor} < 1$ 
and so we need Lemma \ref{lem:nonper}, which says that two long free flights in short succession emanating near a corner point are not possible.
\end{remark}

\begin{lemma}
\label{lem:nonper}
%Assume there are no type 2 corridors. Then
For every $K \in \mathbb N$ there exists $\Cl[Ns]{N:N(K)type3} = \Cr{N:N(K)type3}(K)$ so that
%and for every $N_1 \in \mathbb N$ there is some $N_2= N_2(K, N_1)$ so that
for all $k=1,...,K$,
$$
\mathcal E_{\Cr{N:N(K)type3}} \cap F^{-k}(\mathcal E_{\Cr{N:N(K)type3}}) = \emptyset .
$$
%such that for all $x \in E^+ (h,N_2)$ for some type 3 $x_h$ and for all $k = 1,...,K$
%$f^k(x) \notin \cup_{}$

\end{lemma}

The proof of Lemma \ref{lem:nonper} is longer than the proof of the
other lemmas. To avoid 
disruption in the main ideas here, we postpone this proof to the next section.

Note that in the case of type 2 corridors, two long flights are possible. Specifically, we have

\begin{lemma}
\label{lem:nonper2}
For every $K \in \mathbb N$ there exists $\Cl[Ns]{N:N(K)type2} = \Cr{N:N(K)type2}(K)$ so that for all $k=1,2,...,K$
$$
\mathcal B_{\Cr{N:N(K)type2}} \cap F^{-k} 
(\mathcal E_{\Cr{N:N(K)type3}(K)}\mathcal B_{\Cr{N:N(K)type2}})  = \emptyset
$$
and for all $x \in \mathcal E(x_h,{\Cr{N:N(K)type3}(K)})$, the set
$$
\mathfrak k = \{ k = 1,..., K: F^k(x) \in \mathcal B_{\Cr{N:N(K)type2}} \}
$$
can only be non empty if 
$x_h$ is of type 2. In this case, $\mathfrak k = \{1\}$ or $\mathfrak k = \{2\}$. 
\end{lemma}

\begin{proof}
%This is a consequence of Lemmas \ref{lem:flightgrowth} and \ref{lem:nonper}.
Let 
$$\Cr{N:N(K)type2}(K) = 
\frac{\max_{x \in A} \tau(x)}{\min_{x \in A} \tau(x)}
\Cr{c:conseclongflight}\Cr{N:N(K)type3}^2(K+2),$$ 
where $\Cr{c:conseclongflight}$ is defined in Lemma 
\ref{lem:flightgrowth}.
If 
$x \in \mathcal B(x_h, \Cr{N:N(K)type2})$, then as in Lemma \ref{lem:flightgrowth}, we have either $F^{-1}(x) \in \mathcal E_{\Cr{N:N(K)type3}(K+2)}$ or 
$F^{-2}(x) \in \mathcal E_{\Cr{N:N(K)type3}(K+2)}$. Now the result follows from Lemma \ref{lem:nonper}.

%If $F^k(x) \in \mathcal B_{\Cr{N:N(K)type2}}$ for some $k=3,4,...,K$, then necessarily either 
%$F^{k-1}(x)$ or $F^{k-1}(x)$ is in the set $\in \mathcal E_{\Cr{N:N(K)type3}}$, which is a contradiction with 
%Lemma \ref{lem:nonper}.

\end{proof}

%%%%%%%%%%%%%%%%%%%%%%%%%

For the remaining part of the proof let us fix some $q < 1$.
Recalling \eqref{eq:Zas}, there exists a constant $\bar c$ only depending on $\mathcal D$ so that for any standard pair $\ell = (W, \rho)$,
\begin{equation}
\label{defbrc}
1/ \bar c \leq \mathcal Z_q(\ell)  |W|^q \leq \bar c, \quad 1/ \bar c \leq \mathcal Z_1(\ell)  |W| \leq \bar c.
\end{equation}

%%%%%%%%%%%%%%%%%%%%%%%%

%Recall the definition of $\mathcal E_n$ from \eqref{eq:defE}. 
For a given homogeneous
unstable curve $W$ we write 
$$
\bm T_{N,N'}(W) =\bm T_{N, N',\mathcal D}(W) = \min \{ m \geq 0: F^{-m}(W) \subset \mathcal E_N \mathcal B_{N'}\}.
$$
The next lemma states a weaker version of Theorem \ref{thm2}, namely, when only
those H-components in $F^m(\ell)$ that have not visited $\mathcal E_{N} \mathcal B_{N'}$
for some $N, N'$ large are considered.
%The next lemma states a weaker version of Theorem \ref{thm2}, namely, when only those H-components in $F^m(\ell)$ are considered, which have not visited $\mathcal E_{N} \mathcal B_{N'}$ for some $N, N'$ large.

\begin{lemma}
\label{lem:truncation}
There exists $m_0 \in \mathbb N$, and $\Cl{c:outtacorr}$
so that for every $N, N'$ and for every $K$ there is some $\delta_0 = \delta_0(N,N', K)$ such that 
the following holds for every standard pair $\ell = (W, \rho)$ with $|W| < \delta_0$ and for all 
$m =1,2,..., 2Km_0 + 6$
\begin{equation}
\label{eq:Z01est}
\mathcal Z_1(F^m_*(\ell)|_{W_i: \bm T_{N, N'}(W_i) > m}) < \Cr{c:outtacorr} \mathcal Z_1(\ell),
\end{equation}
\begin{equation}
\label{eq:Z02est}
\mathcal Z_1(F^{K m_0}_*(\ell)|_{W_i: \bm T_{N, N'}(W_i) > Km_0}) < 
2^{-K} \mathcal Z_1(\ell).
\end{equation}
\end{lemma}

\begin{proof}
First, assume that $\rho$ is constant. 
Now we claim %that the proof of Theorem \ref{thm1} implies 
the following: there is some $m_0$ and $\Cr{c:outtacorr}$ so that for any $N, N'$
\begin{equation}
\label{eq:Z11est}
\lim_{\delta \rightarrow 0} \sup_{W: |W|< \delta} \sum_{W_{i} \in F^{m_0}(W): \bm T_{N, N'}(W_i) > m_0} \frac{1}{\Lambda_{i,m_0}} < \frac12,
\end{equation}
and for any $m$,
\begin{equation}
\label{eq:Z12est}
\lim_{\delta \rightarrow 0} \sup_{W: |W|< \delta} \sum_{W_{i} \in F^{m}: \bm T_{N, N'}(W_i) > m} \frac{1}{\Lambda_{i,m}} < \Cr{c:outtacorr}.
\end{equation}
%This claim follows from the proof of Theorem \ref{thm1} as explained below.
To prove this claim, let us replace a small neighborhood (of diameter $< c N^{-1}$)
of the corner points bounding the corridors by a smooth curve so as the 
new billiard table $\mathcal D' = \mathcal D'(N)$ 
contains $\mathcal D$. By construction, for all $W_i$
H-component of $F_{\mathcal D'}^m(W)$ with
$\bm T_{N, N', \mathcal D'}(W_i) > m$ and for all $x \in W_i$, the orbits
$F_{\mathcal D}^{-m}(x), ...., F_{\mathcal D}^{-1}(x), x$ 
and 
$F_{\mathcal D'}^{-m}(x), ...., F_{\mathcal D'}^{-1}(x), x$ coincide.
Then Theorem \ref{thm1} implies 
that the left hand side of
\eqref{eq:Z11est}
is bounded by some number $\beta <1$. Replacing $m_0$ by $m_0^{\ln(1/2) / \ln \beta}$ and using 
\eqref{eq:subadd} and Lemma \ref{lemma:upperbdlen}, we obtain \eqref{eq:Z11est}. 
Although we only proved Lemma \ref{lemma:upperbdlen} 
under the conditions
of Theorem \ref{thm1}, it is valid under the more general conditions of Theorem \ref{thm2}. Indeed, a long flight and an almost grazing collision expands
an unstable curve more than just a long flight.
Likewise, we obtain
\begin{equation}
\label{eq:Z31est}
\lim_{\delta \rightarrow 0} \sup_{W: |W|< \delta} \sum_{W_{i} \in F^{Km_0}: \bm T_{N, N'}(W_i) > Km_0} \frac{1}{\Lambda_{i,Km_0}} < 2^{-K}.
\end{equation}

Also observe that \eqref{eq:Z12est} follows from the proof of Theorem \ref{thm1}
for $m \leq m_0$. Then it also follows for $m \geq m_0$ by
\eqref{eq:Z11est} and by \eqref{eq:subadd}.

Since our construction did not depend on $N'$, it remains to prove that $m_0$ and $\Cr{c:outtacorr}$ are uniform 
in $N$. 
Although the curvature of $\partial \mathcal D'$
is not uniformly bounded in $N$, points visiting the part of the phase space with unbounded curvature are discarded. Then it remains
to observe that the constants $\Xi$ and $C_{\#}$ appearing in the proof of Theorem \ref{thm1} are uniform in $\mathcal D'$ and so is $m_0$ and
$\Cr{c:outtacorr}$.

Now \eqref{eq:Z31est} and \eqref{eq:Z12est} combined with \eqref{eq:densityosc} and \eqref{eq:Zas} imply
\eqref{eq:Z01est} and \eqref{eq:Z02est}. 

Finally, if $\rho$ is not constant, we just need to apply \eqref{eq:densityosc} once more
to complete the proof.
\end{proof}
 
In the setup of Lemma
\ref{lem:truncation}, we discard the points one step 
before reaching $\mathcal E_{N} \mathcal B_{N'}$. 
The next lemma says that we can iterate the map once more and only discard the points upon reaching
$\mathcal E_{N} \mathcal B_{N'}$. Let
$$
\bm T'_{N, N'}(W)  = \min \{ m \geq 1: F^{-m}(W) \subset \mathcal E_N \mathcal B_{N'}\}.
$$

\begin{lemma}
\label{lem:truncation2}
There exists $\Cl{c:intocorr}$
so that for every $K$ and for every large $N, N'$, there is some $\delta'_0 = \delta'_0(N, N', K)$ such that 
the following holds for every standard pair $\ell = (W, \rho)$ with $|W| < \delta'_0$ 
and for all $m =1,2,..., 2Km_0 + 6$
\begin{equation}
\label{eq:Z01est'}
\mathcal Z_1(F^m_*(\ell)|_{W_i: \bm T'_{N, N'}(W_i) > m}) < \Cr{c:intocorr} \mathcal Z_1(\ell)
\end{equation}
\begin{equation}
\label{eq:Z03est'}
\mathcal Z_1(F^{K m_0 + m}_*(\ell)|_{W_i: \bm T'_{N, N'}(W_i) > Km_0 + m}) < 
\Cr{c:intocorr} 2^{-K} \mathcal Z_1(\ell)
\end{equation}
\end{lemma}

\begin{proof}
First we claim that
\begin{equation}
\label{eq:Z03est}
\mathcal Z_1(F^{K m_0 + m}_*(\ell)|_{W_i: \bm T_{N, N'}(W_i) > Km_0 + m}) < 
\Cr{c:outtacorr} 2^{-K} \mathcal Z_1(\ell)
\end{equation}
for
all standard pairs $\ell=(W, \rho)$ with $|W|< (\delta_0)^{3^{Km_0}}$.
Indeed, by Lemma \ref{lemma:upperbdlen}, any H-component $W_i \subset F^{Km_0}(W)$
satisfies $|W_i|< \delta_0$. By applying \eqref{eq:Z01est} to the 
H-components $W_i \subset F^{Km_0}(W)$, \eqref{eq:Z01est} and \eqref{eq:Z02est} imply \eqref{eq:Z03est}.

To derive \eqref{eq:Z01est'}, we write
$$
\mathcal Z_1(F^m_*(\ell)|_{W_i: \bm T'_{N, N'}(W_i) > m})
$$
$$
= \mathcal Z_1(F^m_*(\ell)|_{W_i: \bm T_{N, N'}(W_i) > m}) + 
\mathcal Z_1(F^m_*(\ell)|_{W_i: W_i \subset \mathcal E_{N} \mathcal B_{N'}, \bm T_{N, N'}(W_i) > m}) 
=:Z_{11} + Z_{12}
$$
By \eqref{eq:Z01est}, $Z_{11} \leq \Cr{c:outtacorr} \mathcal Z_1(\ell)$.
Let us write $j \in \mathcal J$ if the 
H-component $W_{j,m-1} \subset F^{m-1}(W)$ contains a point $x \in W_{j,m-1}$ with
$F(x) \in \mathcal E_{N} \mathcal B_{N'}$
and $\bm T_{N, N'}(W_j) > m -1$. Also write $\ell_j = (W_j, \rho_{j,m-1})$.
Choosing $\delta'_0 \leq (\tilde \delta)^{3^{2(Km_0 + 1)}}$ for some $\tilde \delta \leq \delta_0$,
we have $|W_j| \leq \tilde \delta$.

We claim that there is some $\tilde \delta < \delta_0 $ and $\Cl{c:TtoT'}$ so that 
\begin{equation}
\label{eq:TtoT'}
\mathcal Z_1(F_*(\ell_j)) \leq \Cr{c:TtoT'} \mathcal Z_1(\ell_j) \text{ for all }j \in \mathcal J.
\end{equation}

To prove \eqref{eq:TtoT'}, we first claim that there is a constant $C$ only
depending on $\mathcal D$ so that for any $N > \Cr{N:N(K)type3}(1)$ and 
$N' > \Cr{N:N(K)type2}(1)$ fixed, and for any 
$x \in \ell_j$, $\tau(x) < C$. 
Indeed, it is not possible for $x$ to have a long flight in a type 2 or type 3 corridor 
since $x \notin \mathcal E_{N} \mathcal B_{N'}$. 
Let $x$ be so that 
$F(x) \in \mathcal E_{N} \mathcal B_{N'}$.
Then it is also not possible for $x$ to have a long flight
in a type 1 corridor, because in this case $F(x)$ would be close to a boundary point of a type 1 corridor
and we could not have $F(x) \in \mathcal E_{N} \mathcal B_{N'}$. Thus $\tau(x) < C$. This estimate can be extended to all $x \in \ell_j$ by choosing
$\tilde \delta < \delta_0$ (for example, smaller than half of the smallest distance between two distinct points in 
$A = \cup_{corridors} A_H$). Now since the free flight function on $\ell_j$ is uniformly bounded,
\eqref{eq:TtoT'} follows from \cite{DST13}.

We conclude
$$
Z_{12} \leq \sum_{j \in \mathcal J}c_j \mathcal Z_1(F_*(\ell_j)) \leq \Cr{c:TtoT'} 
\sum_{j \in \mathcal J}c_j \mathcal Z_1(\ell_j)$$
$$ \leq \Cr{c:TtoT'} \mathcal Z_1(F^{m-1}_*(\ell)|_{W_j: \bm T_{N, N'}(W_j) > m-1})
\leq \Cr{c:TtoT'} \Cr{c:outtacorr} \mathcal Z_1(\ell).
$$
Thus \eqref{eq:Z01est'} follows with $\Cr{c:intocorr} = \Cr{c:outtacorr} (1+ \Cr{c:TtoT'})$. The derivation of \eqref{eq:Z03est'} 
from \eqref{eq:Z03est} is similar.

\end{proof}

Recall that two long flights are possible in a type 2 corridor, right after one another
or separated by exactly one short flight. 
Our next lemma says that the $\mathcal Z_q$ function can be controlled throughout 
the course of these two long flights.

\begin{lemma}
\label{lem:3steps}
There is some $\Cl{c:threesteps}$ so that
for any standard pair $\tilde \ell = (\tilde W, \tilde \rho)$ with 
$\tilde W \subset \mathcal E_{\Cr{N:N(K)type3}(4)} \mathcal B_{\Cr{N:N(K)type2}(4)}$,
$$
\mathcal Z_q(F^3_* (\tilde \ell)) \leq \Cr{c:threesteps} \mathcal Z_q(\tilde \ell).
$$
\end{lemma}

%%%%%%%%%%%%%%%%%%%%%%%%%%%%%%%%

We finish the proof of the theorem first and then 
will prove Lemma \ref{lem:3steps}. First, we fix a large constant $K$ so that

\begin{equation}
\label{eq:Kcondition}
\Cr{c:intocorr} \bar c ^4 2^{-qK} (1 + (2Km_0) \Cr{c:threesteps} \Cr{c:intocorr} \bar c ^4) 
< \frac12
\end{equation}
holds.
Next, we choose $M = 2Km_0 +6$. Then, we choose 
$\tilde N = \Cr{N:N(K)type3}(M)$, $\tilde N' = \Cr{N:N(K)type2}(M)$. Note that
by Lemma \ref{lem:nonper2}, for every $ x \in \mathcal M$ there exists 
$m =0,...,M$ so that 
\begin{equation}
\label{eq:1visit}
\{ x, F(x), ..., F^M(x) \} \cap \mathcal E_{\tilde N}\mathcal B_{\tilde N'} 
\subset \{ m, m+1, m+2\}.
% \text{ the orbit $x, F(x), ..., F^M(x)$ visits $\mathcal E_{\tilde N}$ at most once. }
\end{equation}

%\begin{itemize}
%\item $\tilde N \geq \Cr{N:N(K)type2}(4)$
%\item for all $m=1,...,M$,
%$$
%\mathcal E_{\tilde N} \cap F^{-m}(\mathcal E_{\tilde N}) = \emptyset 
%$$
%Such $\tilde N$ exists by Lemma \ref{lem:nonper}.
%\end{itemize}

Finally, %given $C_0, N_0', m_0, \Cr{c:intocorr}, K, M$ and $N_1$, 
we fix $\delta'_0 = \delta'_0(\tilde N, \tilde N', K)$ as given by Lemma \ref{lem:truncation2}. 
Then we choose $\delta_1$ so small that for any $W$ with $|W| < \delta_1$, for any $m=1,...,M$, any H-component of $F^m(W)$
is shorter than $\delta'_0$ (e.g., $\delta_1= (\delta'_0)^{3^M}$ works by Lemma \ref{lemma:upperbdlen}). 

We are going to prove Theorem \ref{thm2} with $\varkappa = 1/2$, $\delta = \delta_1$ and $M$ as chosen above.

Note that all standard pairs in the proof are shorter
than $\delta_0'$. Thus, by further reducing $\delta_0'$ if necessary,
we can assume that any standard pair intersecting
$\mathcal E_{\tilde N}\mathcal B_{\tilde N'}$ is fully contained
in $\mathcal E_{\tilde N-1}\mathcal B_{\tilde N'-1}$. To simplify
notations, we will assume that once a standard pair intersects
$\mathcal E_{\tilde N}\mathcal B_{\tilde N'}$, it is also fully 
contained in $\mathcal E_{\tilde N}\mathcal B_{\tilde N'}$.

Let us fix a standard pair $\ell = (W, \rho)$ with $|W| < \delta_1$, let $W_{i,m}$ denote an H-component
of $F^m(W)$ and write
$F^M_*(\ell) = \sum_i c_{i,M} \ell_i$, where $\ell_i = (W_{i,M}, \rho_{i,M})$.
The idea of the proof is now the following. Let the time of the first visit to 
$\mathcal E_{\tilde N}\mathcal B_{\tilde N'}$ be $m$. Then no visit
to $\mathcal E_{\tilde N}\mathcal B_{\tilde N'}$ is possible any time
after $m+2$ by Lemma \ref{lem:nonper2}.
If
$m \leq M/2$, then $M  - m - 3> Km_0$ and so we can use 
\eqref{eq:Z03est'} after the last visit to show that the $Z$ function does not grow. Likewise, if 
$m > M/2$, we will use \eqref{eq:Z03est'} at time zero (before the visit). 

% Let us thus write
%$$ \mathcal Z_q(F^M_*(\ell)) = \sum_i c_{i,M} \mathcal Z_q(\ell_{i,M})$$
To make this idea precise, for all standard pairs $\ell_{i,M}$ we associate 
a set $\mathcal T_{i}$
of integers so that for any $x \in F^{-M} (W_{i,M})$, we have
$F^k(x) \in \mathcal E_{\tilde N}\mathcal B_{\tilde N'}$ for $k=0,1,...,M-1$
if and only if $k \in \mathcal T_i$. By Lemma
\ref{lem:nonper2}, all associated sets $\mathcal T$ can contain up to 2 
integers. Furthermore, if $\mathcal T$ contains exactly 2 numbers, then
$\mathcal T = \{ m, m+1 \}$ or $\mathcal T = \{ m, m+2 \}$ for some 
$m=0,...,M-1$. We have 
$\mathcal Z_q(F^M_*(\ell)) = \sum_i c_{i,M} \mathcal Z_q(\ell_{i,M})$. Now let
$$ Z_m = \mathcal Z_q(F^M_*(\ell) |_{W_i: \min \mathcal T_i =m})$$
and 
$$Z_{M+} = \mathcal Z_q(F^M_*(\ell) |_{W_i:\mathcal T_i =\emptyset}) =
\mathcal Z_q(F^M_*(\ell)|_{W_i: \bm T'_{\tilde N, \tilde N'}(W_i) > M}).
$$
Clearly, 
we have
\begin{equation}
\label{eq:Zmdec}
\mathcal Z_q(F^M_*(\ell)) = \left[ \sum_{m=0}^{M-1} Z_m \right]+ Z_{M+}.
\end{equation}

Given %some non-negative constants $c_i$ with $\sum c_i \leq 1$, 
a substandard family
$\mathcal G = 
(\ell_{\alpha} = (W_{\alpha}, \rho_{\alpha})_{\alpha \in \mathbb N}, \lambda)$ (
recall that substandard means 
$s = \sum_{\alpha=1}^{\infty} \lambda_{\alpha} \leq 1$), we have
\begin{align}
\mathcal Z_q(\mathcal G) = \sum_{\alpha=1}^{\infty} \lambda_{\alpha} \mathcal Z_q(\ell_{\alpha})
& \leq \bar c s  \sum_{\alpha=1}^{\infty} \frac{\lambda_{\alpha}}{s} \frac{1}{|W_{\alpha}|^q} \nonumber \\
& \leq  \bar c s \left[ 
 \sum_{\alpha=1}^{\infty} \frac{\lambda_{\alpha}}{s} \frac{1}{|W_{\alpha}|}
\right]^q 
\leq \bar c^2 s^{1 - q} [\mathcal Z_1(\mathcal G)]^q \nonumber \\
&\leq
 \bar c^2  [\mathcal Z_1(\mathcal G)]^q, \label{eq:Jensen}
\end{align}
where we used \eqref{defbrc} in the first two lines and Jensen's inequality in the second line.
Now combining \eqref{eq:Z03est'} with \eqref{eq:Jensen}, we find
\begin{equation}
\label{eq:ZM+}
Z_{M+} \leq \bar c^2 [\Cr{c:intocorr} 2^{-K} \mathcal Z_1(\ell)]^q
\leq \bar c^2 \left[\Cr{c:intocorr} 2^{-K} \bar c \frac{1}{|W|} \right]^q \leq \Cr{c:intocorr} \bar c^4 2^{-qK} \mathcal Z_q(\ell).
\end{equation}

Next, fix some $m \in [0, M/2]$ and consider the substandard family
$$\mathcal G_{m} = (\ell_{i,m} = (W_{i,m}, \rho_{i,m})_{i\in \mathcal I_m}, \lambda_i = c_{i,m}),$$
where $i \in \mathcal I_m$ if
$$
W_{i,m} \in \mathcal E_{\tilde N} \mathcal B_{\tilde N'}, \bm T'_{\tilde N, \tilde N'}(W_{i, m}) > m.
$$
Note that $\mathcal G_{m}$
corresponds to the image under $F^{m}$ of the points in $W$ whose first hitting
time of the set $\mathcal E_{\tilde N} \mathcal B_{\tilde N'}$ is exactly $m$
and so
$$Z_m = \mathcal Z_q (F_*^{M-m}(\mathcal G_{m})).$$
By \eqref{eq:Z01est'},
$$
\mathcal Z_1(\mathcal G_{M-m}) \leq \Cr{c:intocorr} \mathcal Z_1 (\ell).
$$
Now using \eqref{eq:Jensen} we compute as in \eqref{eq:ZM+} that
\begin{equation}
\label{eq:largem0}
\mathcal Z_q(\mathcal G_{m}) \leq  \bar c^2  [\mathcal Z_1(\mathcal G_{m})]^q
\leq \bar c^2 [\Cr{c:intocorr} \mathcal Z_1(\ell)]^q \leq \Cr{c:intocorr} \bar c^4 \mathcal Z_q(\ell).
\end{equation}
Now fix some $\ell_{i,m} = (W_{i,m}, \rho_{i,m}) \in \mathcal G_{m}$.  
By Lemma \ref{lem:3steps}, we have
\begin{equation}
\label{eq:largem1}
\mathcal Z_q (F^3_*(\ell_{i,m})) \leq \Cr{c:threesteps} 
\mathcal Z_q (\ell_{i,m}).
\end{equation}
Now fix any $W_{j, m+3} \in F^3(W_{i,m})$.  By \eqref{eq:Jensen},
$$
\mathcal Z_q (F^{M-m-3}_*(\ell_{j,m+3})) \leq \bar c^2 
[\mathcal Z_1 (F^{M-m-3}_*(\ell_{j,m+3})) ]^q.
$$
By
\eqref{eq:1visit}, no points of $W_{j, m+3}$ can visit 
$\mathcal E_{\tilde N}\mathcal B_{\tilde N'}$ for $M-m-3$ iterations.
Combining this observation with
the fact that 
$M-m-3 \geq K m_0$ and with \eqref{eq:Z03est'}, we find
\begin{equation}
\label{eq:largem2}
\mathcal Z_q (F^{M-m-3}_*(\ell_{j,m+3})) \leq \bar c^2  
[
\Cr{c:intocorr} 2^{-K} \mathcal Z_1(\ell_{j,m+3}))
]^q
\leq \Cr{c:intocorr} \bar c^4 2^{-qK} \mathcal Z_q(\ell_{j,m+3}).
\end{equation}
Now combining \eqref{eq:largem1} and \eqref{eq:largem2} we find
$$
Z_m = \mathcal Z_q(F^{M-m}_*( \mathcal G_{m})) \leq \Cr{c:threesteps} \Cr{c:intocorr} \bar c^4 2^{-qK}  \mathcal Z_q(\mathcal G_{m}),
$$
and so by \eqref{eq:largem0}
\begin{equation}
\label{eq:largem4}
Z_m \leq \Cr{c:threesteps} (\Cr{c:intocorr})^2 \bar c^8 2^{-qK} \mathcal Z_q(\ell).
\end{equation}

Finally, let us consider the case $m \in [M/2+1, M]$ and define $\mathcal G_{m}$ as before.
%By \eqref{eq:1visit}, we have $\bm T'_{\tilde N, \tilde N'}(W_{i, M-m}) > M-m$ for all
%$\ell_{i,M-m} = (W_{i,M-m}, \rho_{i,M-m}) \in \mathcal G_{M-m}$. 
Since $m > Km_0$, 
we have by \eqref{eq:Z03est'} that
$$
\mathcal Z_1(\mathcal G_{m}) \leq \Cr{c:intocorr} 2^{-K} \mathcal Z_1 (\ell),
$$
and so by \eqref{eq:Jensen},
$$
\mathcal Z_q(\mathcal G_{m}) \leq \Cr{c:intocorr} 2^{-qK} \bar c ^4 \mathcal Z_q (\ell).
$$
Now Lemma \ref{lem:3steps} implies
$$
\mathcal Z_q(F^3_*(\mathcal G_{m})) \leq \Cr{c:threesteps} \mathcal Z_q(\mathcal G_{M-m}) .
$$
Finally, for any $\ell_j \in F^3_*( \mathcal G_{m})$, we combine \eqref{eq:Jensen},
\eqref{eq:1visit} and \eqref{eq:Z01est'}
to conclude
$$
\mathcal Z_q (F_*^{M-m-3}(\ell_j)) \leq \bar c^2 [\mathcal Z_1 (F_*^{M-m-3}(\ell_j))]^q
\leq \bar c^2 [\Cr{c:intocorr} \mathcal Z_1 (\ell_j)]^q 
\leq \Cr{c:intocorr} \bar c^4 \mathcal Z_q (\ell_j).
$$ 
Combining the last three displayed inequalities, we obtain
\begin{equation}
\label{eq:smallm}
Z_m \leq \Cr{c:threesteps} (\Cr{c:intocorr})^2 \bar c^8 2^{-qK} \mathcal Z_q(\ell).
\end{equation}

Now we substite the estimates 
\eqref{eq:ZM+},
\eqref{eq:largem4} and 
\eqref{eq:smallm}
into \eqref{eq:Zmdec}
to conclude
%$$
%\mathcal Z_q(F^M_*(\ell)) 
%$$
\begin{equation}
\label{eq:defK}
\mathcal Z_q(F^M_*(\ell)) 
 \leq \Cr{c:intocorr} \bar c ^4 2^{-qK} (1 + (2Km_0) \Cr{c:threesteps} \Cr{c:intocorr} \bar c ^4) 
\mathcal Z_q(\ell).
\end{equation}
The right hand side of \eqref{eq:defK} is bounded by 
$ \mathcal Z_q(\ell) / 2$ by
\eqref{eq:Kcondition}. 
This completes the proof of Theorem \ref{thm2} assuming Lemma \ref{lem:3steps}. In the remaining part of this
section, we give a proof of this lemma.

\begin{proof}[Proof of Lemma \ref{lem:3steps}]
Let us write $\Cr{N:N(K)type2}= \Cr{N:N(K)type2}(4)$ and 
$\Cr{N:N(K)type3}= \Cr{N:N(K)type3}(4)$.
We will distinguish three cases.

{\bf Case 1}: $\tilde W \subset \mathcal B_{\Cr{N:N(K)type2}}$. 
By Lemma
\ref{lem:expansion}, we have
$$
\mathcal Z_q(F_* (\tilde \ell)) < \Cr{C:type2factor} \mathcal Z_q (\tilde \ell)
$$
and by Lemma \ref{lem:nonper2}, we have
\begin{equation}
\label{eq:nosecondreturn}
(F(\tilde W) \cup F^2(\tilde W)) \cap (\mathcal E_{\Cr{N:N(K)type3}} \mathcal B_{\Cr{N:N(K)type2}})= 
\emptyset.
\end{equation}
Consequently, as in
\eqref{eq:TtoT'},
for any standard pair $\tilde \ell' = (\tilde W', \tilde \rho')$ in the standard family $F_*(\tilde \ell)$, we have
\begin{equation}
\label{eq:ZF^2}
\mathcal Z_1(F^2_*(\tilde \ell')) \leq \Cr{c:TtoT'}^2 \mathcal Z_1(\tilde \ell').
\end{equation}
Using \eqref{eq:Jensen}, we conclude
$$
\mathcal Z_q(F^3_* (\tilde \ell)) \leq \bar c^4 \Cr{C:type2factor} \Cr{c:TtoT'}^2 \mathcal Z_q (\tilde \ell).
$$

{\bf Case 2}: $\tilde W \subset \mathcal E(x_h, \Cr{N:N(K)type3})$ for some type 3 boundary point $x_h$.
By Lemma \ref{lem:expansion0}, we have
\begin{equation}
\label{eq:type3firstexp}
\mathcal Z_1(F_* (\tilde \ell)) \leq \Cr{C:type3factor} \mathcal Z_1 (\tilde \ell).
\end{equation}
As in case 1, \eqref{eq:nosecondreturn} and \eqref{eq:ZF^2} hold. Thus
$
\mathcal Z_1(F^3_* (\tilde \ell)) \leq  \Cr{c:TtoT'}^2  \Cr{C:type3factor} \mathcal Z_1 (\tilde \ell)
$
and so 
$$
\mathcal Z_q(F^3_* (\tilde \ell)) \leq \bar c^4 \Cr{c:TtoT'}^2  \Cr{C:type3factor} \mathcal Z_q (\tilde \ell).
$$

{\bf Case 3}: $\tilde W \subset \mathcal E(x_h, \Cr{N:N(K)type3})$ for some type 2 boundary point $x_h$.
As in case 2, \eqref{eq:type3firstexp} holds.
By Lemma \ref{lem:flightgrowth} and by the choice of ${\Cr{N:N(K)type2}}$, we can write 
$$
F(\tilde W) = \left( \cup_{i \in I_1} \tilde W_i \right) \cup 
\left( \cup_{i \in I_2} \tilde W_i \right)
\cup 
\left( \cup_{i \in I_3} \tilde W_i \right),
$$
where %$\ell_i = (W_i, \rho_i)$, 
$\tilde W_i \subset \mathcal B_{\Cr{N:N(K)type2}}$ for all $i \in I_1$, 
%$\ell_j = (W_j, \rho_j)$ satisfies 
$F(\tilde W_i) \subset \mathcal B_{\Cr{N:N(K)type2}}$ for all $i \in I_2$
and 
$(\tilde W_i \cup F(\tilde W_i)) \cap (\mathcal E_{\Cr{N:N(K)type3}} \mathcal B_{\Cr{N:N(K)type2}})= 
\emptyset$ for all $i \in I_3$.
As in case 1, we derive
$$
\mathcal Z_q(F^2_* (\tilde \ell_i)) < \bar c^4 \Cr{C:type2factor} \Cr{c:TtoT'}^2 \mathcal Z_q (\tilde \ell_i)
$$
for all $i \in I_1$. For $i \in I_2$, we have 
$$
\mathcal Z_1(F_*( \tilde \ell_i)) \leq \Cr{c:TtoT'} \mathcal Z_1( \tilde \ell_i)
$$
as in \eqref{eq:TtoT'}.
Next, for any $\tilde \ell_{i,j} \in F_*(\tilde \ell_i)$ with $i \in I_2$, 
$$
\mathcal Z_q(F_* ( \tilde \ell_{i,j})) < \Cr{C:type2factor} \mathcal Z_q ( \tilde \ell_{i,j})
$$
by Lemma
\ref{lem:expansion}.
Finally, for $i \in I_3$, we have 
$$
\mathcal Z_1(F^2_*( \tilde \ell_i)) \leq \Cr{c:TtoT'}^2 \mathcal Z_1( \tilde \ell_i)
$$
as in \eqref{eq:TtoT'}.

Combining the above estimates, we obtain
$$
\mathcal Z_q(F^3_* (\tilde \ell)) \leq
\bar c ^6 \Cr{C:type3factor} [ 3 \Cr{C:type2factor} \Cr{c:TtoT'}^2 ]
\mathcal Z_q (\tilde \ell).
%\bar c ^2 C_2 \left[ \mathcal Z_1(\cup_{i \in I} \ell_i ) + \mathcal Z_1(\cup_{i \in I} \ell_i ) \right]
$$
The lemma follows with 
$\Cr{c:threesteps}
=3 \bar c ^6 \Cr{C:type3factor}  \Cr{C:type2factor} \Cr{c:TtoT'}^2 
$.

\end{proof}

%%%%%%%%%%%%%%%%%%%%%%%%%%%%%%%%%%%%%%%%%%%%%%

\section{Proof of Lemma \ref{lem:nonper}}
\label{sec:lemma9}

We can assume $\Cr{N:N(K)type3} \geq \Cr{n:type2}$.
Let 
$$
B= \cup_{H \text{ type 3 corridors}} A_H.
$$
First we prove the lemma under simplifying assumptions and then we proceed to more general cases. All important ideas appear in the simplest Case 1, but we need to 
consider several other cases to allow for the singular behavior of the orbit of $B$.\\

{\bf Case 1: No type 2 corridors and
$B \cap R_{-\infty, -2}^{\mathbb H} = \emptyset$.} 
Note that the second assumption means that for
any $x \in B$
and for all $n \geq 1$, $F^n$ is continuous at $F^2(x)$
(recall that the configurational component of $F(x)$ for $x \in B$ is always 
a corner point by definition \eqref{eq:criticalorbit}).

Since the sets $E^+(h,n)$ are disjoint and $B$ is finite, we can find some $\bar N$ so that 
\begin{equation}
\label{defbarN}
%\forall x \in B, \forall k = 1,..,K: F^k(x) \notin \mathcal E_{\bar N}.
%\end{equation}
\{ F^k(x): k=1,...,K, x \in B \} \cap  \mathcal E_{\bar N} =  \emptyset.
\end{equation}

Assume by contradiction that for all $\Cr{N:N(K)type3} \geq \Cr{n:type2}$ 
there is some point $(p_0, \psi_0)$ so that
\begin{equation}
\label{defp0}
(p_0, \psi_0) \in \mathcal E_{\Cr{N:N(K)type3}}, 
(p_i, \psi_i) = F^{i}(p_0, \psi_0):
(p_{\bar k}, \psi_{\bar k}) \in \mathcal E_{\Cr{N:N(K)type3}} \text{ for some }\bar k = 1,...,K.
\end{equation}
As we will see later, 
$p_i$ is not a corner point for any $i=1,...,\bar k$, so the points
$(p_i, \psi_i)$ are uniquely defined.
By passing to a subsequence of positive integers $\Cr{N:N(K)type3}$,
we can assume that $\bar k$ does not depend on $\Cr{N:N(K)type3}$.

Recall Figure \ref{fig1}. Without loss of generality and passing to a 
further subsequence, we can assume that
$(p_0, \psi_0)$ is in a small neighborhood of 
$(r_{H,r,1}, \varphi_{H,r,1})$ (the other 3 cases are similar). Then
the free flight emanating from $(p_0, \psi_0)$ is almost horizontal 
and is in the northeast 
direction. Consequently, $(p_1, \psi_1)$ is in a small neighborhood of 
$(r_{H,l,2}, -\varphi_{H,l,2}) = F(r_{H,l,1}, \varphi_{H,l,1})$
(also recall \eqref{eq:criticalorbit}). 
%See Figure \ref{fig:TODO} {\color{blue} TODO}.
Now let $W \subset \mathcal M$ 
be the line segment between the points 
$(r_{H,l,2}, -\varphi_{H,l,2})$ and $(p_1, \psi_1)$. 
We record for later usage that
the tangent of $W$ satisfies 
$d \varphi / dr \geq 0$. 
Indeed, this follows from the convexity of $\Gamma_{2,2}$.
Also note that $d \varphi / dr = \infty$ is possible
as we can have $p_1 = r_{H,l,2}$. 
Although $W$ may not be an
unstable curve, $F^i(W)$ will be an unstable curve for all 
$i \geq 1$ by the definition of unstable cones.

%\begin{figure}
%\begin{center}

%\caption{...} \label{fig:TODO}
%\end{center}
%\end{figure}

By the second assumption of Case 1, we can find a small $\varepsilon$
so that $F^{K}$ is continuous on the $\varepsilon$ neighborhood
of $F(r_{H,l,2}, -\varphi_{H,l,2}) $. 
Next observe that there is some $\hat N$
so that 
$F^2(\mathcal E((r_{H,r,1}, \varphi_{H,r,1}),\hat N))$
is a subset of the $\varepsilon$ neighborhood
of $F(r_{H,l,2}, -\varphi_{H,l,2})$. 
Now we choose $\Cr{N:N(K)type3} > \max \{ \bar N, \hat N \}$.
Since $\Cr{N:N(K)type3} > \hat N$, $F^{K}(W)$
is a connected homogeneous unstable curve, that is, the trajectory
of $W$ avoids the singularity set $R_0^{\mathbb H}$ up to $K$ 
iterations.
Since $(p_{\bar k}, \psi_{\bar k}) \in F^{\bar k -1}(W)
\cap \mathcal E_{\Cr{N:N(K)type3}}$
and $F^{\bar k}(W)$ is homogeneous, we also have
$F^{\bar k -1}(W)\subset \mathcal E_{\Cr{N:N(K)type3}}$. Thus
$$
F^{\bar k}(r_{H,l,1}, \varphi_{H,l,1}) =
F^{\bar k -1} (r_{H,l,2}, -\varphi_{H,l,2})
\in 
F^{\bar k -1}(W)\subset \mathcal E_{\Cr{N:N(K)type3}}.$$ 
Since 
$\Cr{N:N(K)type3} > \bar N$, this is a contradiction with the choice of $\bar N$. \\

{\bf Case 2 No type 2 corridors and for all $x \in B$, and for all $n \geq 2$, $\Pi_{\mathcal D} F^n(x)$ is a regular point.}
The second assumption means that the future trajectory of points
in $B$, after the first collision, can contain grazing collisions but no
corner points.

We use the same idea as in the proof of
Case 1 with the main difference that we need to study the flow
instead of the map as the latter may not be continuous at 
$F(r_{H,l,2}, -\varphi_{H,l,2}) $. To derive the contradiction,
we need to consider possibly different collision times $\bar k$
and $\hat k$ of the points $(p_0, \psi_0)$ and $(r_{H,l,1}, \varphi_{H,l,1})$.

First, we update the definition of $\bar N$ by replacing $k=1,...,K$
by $k=1,...,L$ in \eqref{defbarN}, where $L$ is a 
constant only depending on $\mathcal D$ and $K$ so that for any $x \in B$,
the first $L$ collisions in the future contain at least $K+1$ non-grazing collisions:
\begin{equation}
\label{defbarN2}
%\forall x \in B, \forall k = 1,..,L: 
\{ F^k(x): k=1,...,L, x \in B \} \cap  \mathcal E_{\bar N} =  \emptyset.
\end{equation}
To prove that such an $L$ exists, observe that for a given $K$ there is a $\delta >0$
so that the orbit of $B$ up to $K$ collisions avoids the $\delta$ neighborhood
of the type 1 boundary points (as all type 1 boundary points are invariant
under the billiard map). Thus all free flight before the next $K$
collisions is bounded by some constant $\tau_K$ on $B$. Also note that two grazing collisions are necessarily
separated by a time $\tau_*$ (this is obvious in case the collisions happen on 
different scatterers, and by \cite{C99}, corner sequences can only contain one
grazing collision). Thus we can choose $L = (K+1)\tau_K/\tau_*$.

Next we claim that there is some $\varepsilon$ so that
\begin{equation}
\label{defeps}
dist(\cup_{k=1}^L F^k(B), \mathcal E_{\bar N}) > \varepsilon.
\end{equation}
%the distance between the set $F^{k}(B)$, $k=1,...,L$
%and $\mathcal E_{\bar N}$ is bigger than $\varepsilon$. 
To prove \eqref{defeps}, first note that for every $x \in B$,
$F^{k}(x)$ cannot be in $\mathcal E_{\bar N}$ by the definition 
of $\bar N$. Furthermore, 
$F^{k}(x) \notin \partial \mathcal E_{\bar N}$ for $x \in B$
because otherwise 
$\Pi_{\mathcal D}F^{k}(x)$ would be a corner point, contradicting the assumptions of Case 2. Thus the desired 
$\varepsilon >0$ exists.

Since the forward orbit 
of
$F(r_{H,l,2}, -\varphi_{H,l,2})$
only contains regular points, for any $t \geq 0$,
the flow $\Phi^t$ is continuous
at the point $F(r_{H,l,2}, -\varphi_{H,l,2}) $ (recall that $\mathcal M$ 
is identified with a subset of $\Omega$). The continuity of the flow
at points whose orbit avoids corner points follows from \cite[Exercise 2.27]{CM06}.
%is proved in \cite{CM06} for type A billiards, but the proof is valid here
%locally as the orbit avoids corner points.

Let $\mathcal T_k = \sum_{\ell=0}^{k -1} \tau(F^\ell(r_{H,l,2}, -\varphi_{H,l,2}))$.
Let $\mathcal P$ be the configurational component of the
trajectory of $F(r_{H,l,2}, -\varphi_{H,l,2})$
under the flow in time $\mathcal T_L$.
Let $\bar \varepsilon \in (0, \varepsilon)$ be so small that any 
trajectory of the flow up to time $\mathcal T_L$ that stays
$\bar \varepsilon$ close to $\mathcal P$ can only have a collision
with angle $\varphi$, $\cos (\varphi) > \bar \varepsilon$ whenever 
$\bar \varepsilon$ close to a non-grazing collision of $\mathcal P$.

Next we claim the following. There is some $\hat N$ large enough
so that for any $\Cr{N:N(K)type3} > \hat N$ and for any given $(p_0, \psi_0)$
satisfying \eqref{defp0} we can find $\hat k \leq L$ so that
$(p_{\bar k}, \psi_{\bar k})$ is $\bar \varepsilon$
close to $F^{\hat k}(r_{H,l,2}, -\varphi_{H,l,2})$.
To prove this claim, observe that 
the continuity of $\Phi^{\mathcal T_L}$
at $F(r_{H,l,2}, -\varphi_{H,l,2})$ implies the existence of $\hat N$
so that if $(p_0, \psi_0) \in \mathcal E_{\hat N}$,
and so $(p_2, \psi_2)$ is close to 
$F(r_{H,l,2}, -\varphi_{H,l,2})$, then the trajectories of these
two points remain
$\bar \varepsilon$ close up to $\mathcal T_L$ flow time. It remains
to define $\hat k$ as the number of 
collisions in the orbit
of $F(r_{H,l,2}, -\varphi_{H,l,2})$ before flow time $\mathcal U + 
\bar \varepsilon$, where $\mathcal U = 
\sum_{\ell=1}^{\bar k -1} \tau(p_\ell, \psi_\ell)$.
By construction, $F(r_{H,l,2}, -\varphi_{H,l,2})$ experiences 
at most $\bar k$ non-grazing collisions before flow time $\mathcal U + \bar \varepsilon$ and so 
$\hat k \leq L$ holds.

We choose $\Cr{N:N(K)type3}> \max \{ \bar N, \hat N\}$ as before. 
We derived that $(p_{\bar k}, \psi_{\bar k})$ is $\bar \varepsilon$
close to $F^{\hat k}(r_{H,l,2}, -\varphi_{H,l,2})$. Now recall from \eqref{defeps} that
$F^{\hat k}(r_{H,l,2}, -\varphi_{H,l,2})$ is not in the 
$\varepsilon$ neighborhood of $\mathcal E_{\bar N}$. This
contradicts \eqref{defp0}.\\

{\bf Case 3 No type 2 corridors and for all $x \in B$, and for all $n \geq 2$, $\Pi_{\mathcal D} F^n(x)$ is not a boundary point of a corridor.}
The future trajectories
of points in $B$, after the first collision, are now allowed to contain
corner points, but they cannot return to $B$.

An important observation is that the orbit of $B$ cannot contain type 1 boundary points.
Indeed, the preimage of any type 1 boundary point is itself.

Recall that now the billiard (both flow and map) can be multivalued
at $F(r_{H,l,2}, -\varphi_{H,l,2})$.
We say that a multivalued map $T$ is multicontinuous at $x$ 
if for every $\varepsilon$ there is some $\delta$ so that
for any $y$ with $dist(x,y) < \delta$ there is a mapping
$g_{x,y}$ from $T(y)$ to $T(x)$ (recall that these are sets now!)
so that for any $z \in T(y)$, $dist(z, g_{x,y}(z)) < \varepsilon$. 
By our definition of the billiard at corner points and by the assumption
of Case 3, $\Phi^t$ is multicontinuous at $F(r_{H,l,2}, -\varphi_{H,l,2})$
for any $t \geq 0$.
Indeed, the values of the flow were defined as the possible
limit points of nearby regular trajectories. 
%Now the proof of
%\cite[Lemma 2.24]{CM06} applies as the
%finite horizon condition is replaced by the assumption of Case 3.

Now observe that \eqref{defbarN2} is still valid. 
Furthermore, \eqref{defeps} also remains true 
but requires a new proof. 
Since the future orbit of $B$ can contain corner points, a priori
it may be possible that 
$F^k(x)\cap \partial \mathcal E_{\bar N} \neq \emptyset$. However,
the intersection is finite and cannot contain points in $B$
by the assumption of Case 3. Thus for any point $y  
\in F^k(x)\cap \partial \mathcal E_{\bar N}$ there is some $h$ and
$N(y)$ so that $y \in \partial E^+(h,N(y))$. Since there
are finitely many points $y$, we still can guarantee \eqref{defeps}
by choosing $\bar N > \max_y N(y)$.

Now we can conclude the proof of the lemma as in Case 2 with the 
only difference that we use {\it one} element of each of 
the sets $F^{\hat k}(r_{H,l,2}, -\varphi_{H,l,2})$ 
and $F^{\bar k} (p_0, \psi_0)$
to derive the contradiction.\\

{\bf Case 4 No type 2 corridors}
The difference from Case 3 is that now points in $B$ are allowed 
to return to $B$ under branches of iterates of $F$. 

%We can repeat the proof in Case 3 with one major difference. 
First, observe that \eqref{defbarN2} is still valid.
Indeed, if $F^{k_1}(x_1) = x_2$ for some $x_1, x_2 \in B$,
then by definition $x_2 \notin \mathcal E_{\Cr{N:N_0}}$ as $x_2$
does not have long free flight.
In Case 4, \eqref{defeps} is not true as $F^k(B)$ may intersect 
with $B$, and for any positive integer $N$,
$B \subset \partial \mathcal E_N$.
However, we have the weaker statement
$$
\dist(\cup_{k=1}^L F^k(B) \setminus B, \mathcal E_{\bar N}) < \varepsilon
$$
which is proved exactly as in Case 3.

Consider one element of the set 
$F^{\bar k} (p_0, \psi_0) \cap \mathcal E_{\Cr{N:N(K)type3}}$. Denote this 
point by $(p_{\bar k}, \psi_{\bar k})$. Once
$(p_{\bar k}, \psi_{\bar k})$ is fixed, we can find a
unique
sequences of points $(p_{k}, \psi_{k})$ so that $(p_{k}, \psi_{k})
\in F(p_{k-1}, \psi_{k-1})$ for $k=1,...,\bar k$. 
Likewise, we can find a unique sequence of points 
$y_1,...,y_{k_1}$ so that $y_1 = (r_{H,l,2}, -\varphi_{H,l,2})$,
$y_k \in F^{s_k}(y_{k-1})$ and $y_k$ is $\bar \varepsilon$ close
to $(p_{k}, \psi_{k})$. This is the same construction as in Case 3;
the only difference is that we can only go up to iterate 
$$k_1 = \min\{ k: y_k \in B\} \wedge \hat k.$$
Indeed, the construction of Case 3 works up to the first time
$y_k \in B$.
If $\hat k < \min\{ k: y_k \in B\}$, then the proof is completed
as in Case 3. 
Assume now that $\hat k = \min\{ k: y_k \in B\}$. Recalling 
the definition of $W$ from Case 1, we see that
$y_{k_1}$ and $(p_{k_1}, \psi_{k_1})$ are on a short unstable
curve. 
That is, the tangent of the line segment
$(y_{k_1},(p_{k_1}, \psi_{k_1}))$ satisfies $d \varphi/dr \geq 0$. But 
note that for any point $ z \in \mathcal E(y_{k_1},\Cr{N:N(K)type3})$ the tangent
of the line segment $(y_{k_1},z)$ satisfies $d \varphi/dr < 0$ (see one case on
the left panel of Figure \ref{fig2}. 
In the other case, i.e. when $\Pi_{\mathcal D}y_{k_1}$
is the right endpoint of $\Gamma_{ij}$, the region $\mathcal E(y_{k_1},\Cr{N:N(K)type3})$ is to the northwest from $y_{k_1}$.) 
This means that $(p_{k_1}, \psi_{k_1})$ cannot be in 
$\mathcal E_{\Cr{N:N(K)type3}}$ which is a contradiction.
Now assume that $\hat k > k_1$. 
Assume that $y_k = (r_{H,r,1}, \varphi_{H,r,1})$ 
(the other three cases are similar). Then define 
$y_{k_1+1} = (r_{H,r,2}, - \varphi_{H,r,2})$ and 
$(p_{k_1 + 1}, \psi_{k_1 + 1}) = 
F(p_{k_1}, \psi_{k_1})$ (which exists uniquely as $\bar \varepsilon$
is small). Furthermore, the line segment $W_1$ joining
$y_{k_1+1} $ and $(p_{k_1 + 1}, \psi_{k_1 + 1})$ has a 
tangent that satisfies $d \varphi / dr \geq 0$. Then we can repeat  
the previous construction with $(p_0, \psi_0)$ replaced by  $(p_{k_1 + 1}, \psi_{k_1 + 1})$
and $y_1$ replaced by $y_{k_1+1}$. Let 
$$k_2 = \min \{ k > k_1 : y_k \in B \} \wedge \hat k.$$
If $\hat k \leq  \min \{ k > k_1 : y_k \in B \}$, then the proof is completed as before.
If $\hat k > \min \{ k > k_1 : y_k \in B \}$, then we can define $(p_{k_2+1}, \psi_{k_2+1})$, $y_{k_2 + 1}$
as before. Following this procedure inductively, we consider as many $k_\ell$'s as needed. For some
$\ell < K$, we will have $\hat k < k_{\ell}$ whence we can finish the proof.\\

{\bf Case 5 Some type 2 corridors}

When type 2 corridors are allowed, the previous proof can be repeated with minor 
changes, which we list next.
First, we replace $B$ by $A'$, where $A'$ is defined in \eqref{def:A'}.
Let $(p_0, \psi_0) \in \mathcal E(x_h, \Cr{N:N(K)type3})$ as before. If $x_h \in B$, we proceed as before.
Assume now that $x_h  = (r_h, \varphi_h) \in A' \setminus B$.
Then the boundary points of the corresponding corridor $H$ are
$$
A_H = \{ x_h  = (r_h, \varphi_h), x_{h'}  = (r_{h'}, \varphi_{h'}), (r_{h''}, - \pi/2), (r_{h''}, \pi/2) \}.
$$
Here, $r_h$ and $r_{h'}$ correspond to the corner point on one side of the corridor,
and the points $(r_{h''}, \pm \pi/2)$ correspond 
to the regular boundary point of the corridor.
As in Lemma \ref{lem:flightgrowth}, we find that
either $F^2(p_0, \psi_0)$ or $F^3(p_0, \psi_0)$ is in a small neighborhood of $(r_{h'}, -\varphi_{h'})$.
Indeed, the particle starting from $(p_0, \psi_0)$ experiences a long free flight, after which it collides 
once or twice in a small neighborhood of the regular boundary point of the corridor, and then has another long free flight
terminating in a small neighborhood of the boundary corner point.
Replacing $(p_1, \psi_1)$ with either $F^2(p_0, \psi_0)$ or $F^3(p_0, \psi_0)$ (whichever is close to $(r_{h'}, -\varphi_{h'})$),
and making similar adjustment at all times $k_1, k_2,...,k_{\ell}$ as introduced in Case 4, we can repeat the proof of Case 4.

%%%%%%%%%%%%%%%%%%%%%%%%

\section{Proof of Theorem \ref{thm3}}
\label{sec:thm3}

Theorem \ref{thm3} is quite intuitive. Indeed, conditions (A1) and (A2)
prescribe degeneracies in the geometry which can be easily 
destroyed by a small perturbation (e.g., the genericity of (A1) was stated
in \cite{SzV07} without a proof). It is not difficult to turn this intuition
into a rigorous proof, but we decided to include such a proof for completeness.

Let us fix some combinatorial data $(d, J_1, ..., J_d)$.
Since $\bm D$ is a disjoint union of the open sets $\bm D_{d, J_1, ..., J_d}$, it is enough to
prove the theorem for $\bm D_{d, J_1, ..., J_d}$. To simplify the notation, we will drop
the subscript and only write 
$\bm D$ instead of $\bm D_{d, J_1, ..., J_d}$ in the sequel.

We say that an {\it incipient corridor} $H$ 
is a direction $v= v_H \in [0, \pi)$ 
and a 
connected subset $Q_H$ of $\mathcal D$ with empty interior
satisfying \eqref{eq:QH}.
The difference between corridors and incipient corridors is that in case of the
latter one, $Q_H$ has empty interior.
That is, the configurational component of an infinite orbit
that only experiences grazing collisions, but does so on both sides of the flight, constitutes an incipient corridor.

Now define the set $\bm D_0 \subset \bm D$ as the set of billiard tables $\mathcal D$
that satisfy (A1) and (A2) and do not have incipient corridors. We are going
to prove that $\bm D_0$ is open and dense. This clearly implies the theorem.\\

{\bf Step 1: $\bm D_0$ is open} 

Given $\mathcal D \in \bm D_0$,
we need to find $\varepsilon >0$ so that $U$, the $\varepsilon$ neighborhood of 
$\mathcal D$, is contained in $\bm D_0$. 
For $\mathcal D \in \bm D_0$, let $\kappa_+$ denote the maximal curvature at regular
points. Then $\mathbb T^2 \setminus \mathcal D$ contains a disc of radius 
$\kappa_+^{-1}$. By choosing $\varepsilon < \kappa_+^{-1}/2$, we ensure that for
all $\mathcal D' \subset U$ there is a disc of radius $\kappa_+^{-1}/2$ inside
$\mathbb T^2 \setminus \mathcal D'$. 

Next we claim that there is a finite set 
$\mathcal V \subset \mathcal S^1$ so that for any 
$\mathcal D' \subset U$ and for any corridor $H$ on $\mathcal D'$, the direction of $H$ satisfies $v_H
\in \mathcal V$.
To prove the claim, first observe that for any direction $v_H$, $\tan v_H$ is rational. Indeed,
if it was not rational, then the set $\{q+tv_H\}_{t \in \mathbb R}$ would be
dense in $\mathbb T^2$. Now assume that $v_H \in [0, \pi/4]$ (the other cases
are similar). Let us write $\tan v_H = P/Q$ where $0 < P < Q$ are coprime
integers. Then necessarily $Q < 3 \kappa_+$ because otherwise 
the set $\mathbb T^2 \setminus \{q+tv_H\}_{t \in \mathbb R}$ would
not contain a 
ball of radius $\kappa_+^{-1}/2$. The claim follows. 

Let $\mathcal V_0 \subset \mathcal V$ be the set of directions in which there is a
corridor on $\mathcal D$ and 
let $v \in \mathcal V \setminus \mathcal V_0$.
Now we claim that there is some $\delta_v >0$ so that for any $q \in \mathbb T^2$,
the line $q + t v$, $t \in \mathbb R$ intersects with the complement
of the $\delta_v$ neighborhood of $\mathcal D$. Indeed, this 
follows from the assumption that $\mathcal D$
does not have incipient corridors and from compactness. 
Likewise, for any $v \in \mathcal V_0$ there is some 
$\delta_v >0$ so that
for any 
$q \notin \cup_{H: v_H = v}\bm B_{\delta_v}(Q_H)$
the line $q + t v$, $t \in \mathbb R$ intersects with the complement
of the $\delta_v$ neighborhood of $\mathcal D$. 
Here, $\bm B_\rho (Q)$ means the $\rho$ neighborhood of $Q \subset \mathbb T^2$.
Further reducing $\varepsilon$ if necessary,
we can assume that $\varepsilon < \delta_v$ for all 
$v \in \mathcal V$. Consequently, for all $\mathcal D' \in
U$, there is an injection from the set of corridors of $\mathcal D'$
to the set of corridors of $\mathcal D$ preserving the angle of the corridors. 
Indeed, by the choise of $\varepsilon$,
no new corridor can open up if we perturb $\mathcal D$ with an $\varepsilon$
small $\mathcal C^3$ (in fact $\mathcal C^0$) perturbation. It may be 
possible at this point that some corridors close during the perturbation,
which we rule out next.

Now since $\mathcal D$ satisfies (A1), the following is true.
For any corridor $H$ on $\mathcal D$, we can find some $\varepsilon_H >0$
%and $\delta_H >0$ 
so that for any $q$ in the $\varepsilon_H$ neighborhood
of $Q_H$, 
$$
\{ q+ t v_H: t \in \mathbb R \} \cap (\mathbb T^2 \setminus \mathcal D)
\subset \bm B_{\varepsilon_H} (B_H).
$$
Here, $B_H = \partial Q_H \cap \partial \mathcal D$
has two elements by (A1). Further
reducing $\varepsilon$ as necessary, we can assume $\varepsilon < \varepsilon_H/2$
for all corridors $H$ on $\mathcal D$. Now by construction
for any corridor $H$ on $\mathcal D$
and for any $\mathcal D' \in U$, we can find a correspoding corridor $H'$ on 
$\mathcal D'$ so that $v_H = v_{H'}$ and the symmetric difference of $Q_H$ and
$Q_{H'}$ is contained in the $\varepsilon$ neighborhood of the boundary of $Q_H$. In particular, the injection constructed in the previous paragraph is now a bijection. Furthermore, $B_{H'} = \partial Q_{H'} \cap \partial \mathcal D'$ has two 
elements. We conclude that $\mathcal D'$ satisfies (A1) and has no incipient corridors.

Finally, since $\mathcal D$ satisfies (A2), there is some angle $\alpha >0$
so that for any type 2 or 3 corridor $H$ and for any boundary
corner point $q_H \in B_H$,
the angle between $v_H$ and any one-sided tangent to $\partial \mathcal D$
at $q_H$ is bigger than $\alpha$. Further reducing $\varepsilon$ if necessary, we can assume
$\varepsilon  < \alpha/2$. This guarantees that all $\mathcal D' \in U$ satisfy
(A2). It follows that $\bm D_0$ is open.\\

{\bf Step 2: $\bm D_0$ is dense} 

We will need the following simple lemma.

\begin{lemma} (Local enlargement)
Let $\mathcal D$ be an admissible billiard table, $ q \in \partial \mathcal D$ and $\varepsilon >0$. Then there exists
another admissible billiard table $\tilde{\mathcal D}$
so that 
\begin{itemize}
\item $d(\mathcal D, \tilde{ \mathcal D}) < \varepsilon$
\item $\mathcal D$ and $ \tilde{ \mathcal D}$ coincide on the complement of the $\varepsilon$ neighborhood of $q$
\item $\tilde{ \mathcal D } \subset \mathcal D$  with $q$ being in the interior of $\mathbb T^2 \setminus \tilde{ \mathcal D}$
\end{itemize}
\end{lemma}

\begin{proof}
Assume that $q \in \Gamma_{i,j}$ is a regular point. Then we can represent 
$\Gamma_{i,j}$ 
in a small neighborhood of $q$ in local coordinates as a graph of a concave function 
$f:[-1,1] \rightarrow \mathbb R^2$ with $f(0) = 0$. Fix a $C^\infty$ function $\phi: \mathbb R \rightarrow \mathbb R$ 
so that $\phi(0) = 1$ and $\phi$ is identically zero outside of $(-1/2, 1/2)$. Let the curvature of $\Gamma$ at $q$
be $\kappa$ and $\varepsilon' = \min\{ \varepsilon, \kappa\} /(10 \| \phi\|_{\mathcal C^3})$. 
Now define $\tilde{ \mathcal D}$ to be the same as $\mathcal D$ except that the image of $f$ is replaced by the image of
$\tilde f = (1 + \phi) f$. By construction, $\tilde{ \mathcal D}$ is an admissible table satisfying the requirements.

The case of corner points is similar, we just need to perturb both curves meeting at the corner point.
\end{proof}

To prove that $\bm D_0$ is dense, fix an arbitrary $\mathcal D \in \bm D$ and 
$\varepsilon >0$. We need to find some $\hat{ \mathcal D} \in \bm D_0$ with
$d(\mathcal D, \hat{\mathcal D}) < \varepsilon$.
In the remaining part of the proof, the term corridor can stand for either non-incipient or
incipient corridor.

Let us denote by $U$ the $\varepsilon$ neighborhood of $\mathcal D$.
Reducing $\varepsilon$ if necessary, we can assume as in Step 1 that there is
a finite set
$\mathcal V$ so that for any $\mathcal D' \in U$ and for any corridor $H$ 
on $\mathcal D'$, $v_H \in \mathcal V$. Furthermore, for any given $v \in V$, there may only
be a bounded number of corridors with direction $v$. Let us fix some ordering of the 
corridors. E.g. fix arbitrary ordering on $\mathcal V$ and define $H_1 < H_2$ 
if $v_{H_1} < v_{H_2}$.
For corridors $H_1$, 
$H_2$ with $v_{H_1} = v_{H_2}$, project $Q_{H_i}$ to the direction perpendicular to $v_{H_1}$ (when $\mathbb T^2$ is identified with the unit square). If the projections
are denoted by $\pi Q_{H_1}$, $\pi Q_{H_2}$, then define $H_1 < H_2$ if
the origin is closer to $\pi Q_{H_1}$ than to $\pi Q_{H_2}$. 

We are going to consider billiard tables $\mathcal D' \in U$ with
${\mathcal D' } \subset \mathcal D$. This guarantees 
that no new corridors open up by the perturbation, that is there is an injection
$\iota_{\mathcal D'}$
from the set of corridors on ${\mathcal D'}$ to the set of corridors on 
${\mathcal D}$ that preserves the angle and the ordering. 
Note however that this time $\iota$ may not be a bijection
as we want to eliminate incipient corridors.
Let $H_1< H_2< ... <H_k$ be the 
ordered list of corridors of $\mathcal D$. 
Let us say that a corridor on a billiard table $\mathcal D'$ is good if it is non-incipient
and does not violate (A1) and (A2). 

We are going to define 
$\mathcal D = \mathcal D'_0, \mathcal D'_1,..., \mathcal D'_k = \hat{\mathcal D}$
in a way that for every $i=1,...,k$,
\begin{itemize}
\item $d(\mathcal D'_i, \mathcal D'_{i+1}) < \varepsilon /2k$
\item the corridors in $\iota^{-1}_{\mathcal D'_{i}}(\{ H_1, ..., H_i \})$
are all good.
\end{itemize}
If these items can be guaranteed, then it follows that $\hat{\mathcal D} \in \bm D_0$
and $d(\mathcal D, \hat{\mathcal D}) < \varepsilon$,
which completes the proof.
We prove the above items by induction. Assume they hold for $i$. 
If $\iota^{-1}_{\mathcal D'_{i}}(\{ H_{i+1} \}) = \emptyset$, then we define
$\mathcal D'_{i+1} = \mathcal D'_{i}$. Next assume that there is a corridor
$H'$ on $ \mathcal D'_{i}$ with $\iota_{\mathcal D'_{i}}(H') = H_{i+1}$.
If $H'$ is good, then we define
$\mathcal D'_{i+1} = \mathcal D'_{i}$. Let us now assume that 
$H'$ is either incipient or violates (A1) or (A2). In all cases, we can apply the 
local enlargement lemma with $\mathcal D, \varepsilon$ replaced by $\mathcal D_i,
\delta_{i+1} < \varepsilon/2k$ at some point $q_{i+1}$
to produce another billiard table $\mathcal D'_{i+1}$ with 
either $\iota^{-1}_{\mathcal D'_{i+1}}(\{ H_{i+1} \}) = \emptyset$ (in case $H'$ was incipient) or $\iota^{-1}_{\mathcal D'_{i+1}}(H_{i+1})$ is a good corridor.
Indeed, if $H'$ is incipient, then we apply the local enlargement lemma at 
a point $q_{i+1} \in H' \cap \partial \mathcal D'_{i}$. If $H'$ violates (A1), then it has several boundary points
on at least one of its sides. Now we apply the local enlargement lemma at one of 
these boundary points.
 Finally, if $H'$ violates (A2), then we apply the 
local enlargement lemma at the given boundary corner point. Clearly, the perturbation
can be made in a way that the direction of the half-tangents is modified and so
$\iota^{-1}_{\mathcal D'_{i+1}}(H_{i+1})$ will not violate (A2). 
%In case we need
%to use the local enlargement lemma on both sides of the corridor, to keep the notations
%simple and with a slight abuse of notation, we denote both points by $q_{i+1}$.

Finally, we claim that by choosing $\delta_{i+1}$ small, we can guarantee that
the corridors in 
$\iota^{-1}_{\mathcal D'_{i+1}}(\{ H_1,...,H_{i} \})$ are all good, too. 
Note that this is not entirely obvious as a corner point
can be on the boundary of multiple corridors (with different directions)
and so the perturbation at iteration $i+1$ may change 
$\iota^{-1}_{\mathcal D'_{i}}(H_j)$ with $j\leq i$.
However, Step 1 ensures that there is some $\delta_{i+1} \in (0, \varepsilon / 2k)$ so that $\delta_{i+1}$
small $\mathcal C^3$ perturbations preserve the goodness of corridors. This
completes the poof of the induction. It follows that $\bm D_0$ is dense.

%%%%%%%%%%%%%%%%%%%%%%%%%%%

\section*{Acknowledgement}
We thank Jacopo de Simoi and Imre P\'eter T\'oth for useful discussions.
This work started when PN was a Brin fellow at University of Maryland,
College Park.
MB was partially supported by NSF DMS 1800811 and
PN was partially supported by NSF DMS 1800811 and NSF DMS 1952876.

\end{document}